\documentclass[11pt,reqno]{amsart}

\usepackage{amscd, amsfonts, amsmath, amssymb, amsthm, mathrsfs, appendix, enumerate, epsfig, tabu, graphicx, hyperref, pdfpages, tikz, mathtools, comment, tikz-cd,pinlabel,bbm,subcaption}
\usepackage[all,cmtip]{xy}
\makeatletter
\@namedef{subjclassname@2020}{Mathematics Subject Classification \textup{2020}}

\newsavebox{\@brx}
\newcommand{\llangle}[1][]{\savebox{\@brx}{\(\m@th{#1\langle}\)}%
	\mathopen{\copy\@brx\kern-0.5\wd\@brx\usebox{\@brx}}}
\newcommand{\rrangle}[1][]{\savebox{\@brx}{\(\m@th{#1\rangle}\)}%
	\mathclose{\copy\@brx\kern-0.5\wd\@brx\usebox{\@brx}}}

\makeatother

\newtheorem{theorem}{Theorem}[section]
\newtheorem{corollary}[theorem]{Corollary}
\newtheorem{lemma}[theorem]{Lemma}
\newtheorem{proposition}[theorem]{Proposition}

\theoremstyle{definition}

\newtheorem{definition}[theorem]{Definition}
\newtheorem{example}[theorem]{Example}
\newtheorem{remark}[theorem]{Remark}
\numberwithin{equation}{subsection}
\newtheorem*{ack}{Acknowledgement}

\newcommand{\Aut}{\operatorname{Aut}}
\newcommand{\SR}{\mathbf{SR}}
\newcommand{\SQ}{\mathbf{SQ}}

\newcommand{\Conj}{\operatorname{Conj}}
\newcommand{\Core}{\operatorname{Core}}

\newcommand{\C}{\mathbf{C}}

\newcommand{\Ab}{\mathbf{Ab}}
\newcommand{\T}{\mathrm{T}}
\newcommand{\e}{\epsilon}
\newcommand{\Sa}{\textrm{S}}

\newcommand{\Hom}{\operatorname{Hom}}

\newcommand{\id}{\mathrm{id}}

\newcommand{\ie}{\mathrm{i.e.}}

\begin{document}
	\setcounter{section}{0}
	\title[Generalized (co)homology of symmetric quandles]{Generalized (co)homology of symmetric quandles over homogeneous Beck modules}
	\author{Biswadeep Karmakar}	
	\author{Deepanshi Saraf}
	\author{Mahender Singh}
	
	\address{Department of Mathematical Sciences, Indian Institute of Science Education and Research (IISER) Mohali, Sector 81,  S. A. S. Nagar, P. O. Manauli, Punjab 140306, India.}
	\email{biswadeep.isi@gmail.com}
	\email{saraf.deepanshi@gmail.com}
	\email{mahender@iisermohali.ac.in}
	
	
	
\begin{abstract}
A quandle equipped with a good involution is referred to as symmetric. It is known that the cohomology of symmetric quandles gives rise to strong cocycle invariants for classical and surface links, even when they are not necessarily oriented. In this paper, we introduce the category of symmetric quandle modules and prove that these modules completely determine the Beck modules in the category of symmetric quandles. Consequently, this establishes suitable coefficient objects for constructing appropriate (co)homology theories. We develop an extension theory of modules over symmetric quandles and propose a generalized (co)homology theory for symmetric quandles with coefficients in a homogeneous Beck module, which also recovers the symmetric quandle (co)homology developed by Kamada and Oshiro [Trans. Amer. Math. Soc. (2010)]. Our constructions also apply to symmetric racks. We conclude by establishing an explicit isomorphism between the second cohomology of a symmetric quandle and the first cohomology of its associated group.
\end{abstract}
\maketitle

\section{Introduction}
Within the realm of oriented knot theory, racks and quandles \cite{Joyce1979, MR0672410}, and their cohomological ramifications  \cite{MR1990571, MR1364012, MR2255194}, have undoubtedly played a remarkable role. Their rich theories, combined with various refinements, have given rise to a plethora of connections spanning Hopf algebras \cite{MR1994219}, mapping class groups \cite{SDR2023, MR2699808}, Riemannian symmetric spaces \cite{MR0217742}, the Yang--Baxter equation \cite{MR3558231}, Yetter--Drinfeld modules \cite{MR1994219},  and contact geometry \cite{KSS2023}, to name a few.
\par

In knot theory, the existence of an orientation often becomes a crucial prerequisite, particularly in scenarios where one intends to work with the quandle cocycle invariants. The challenge of overcoming the need for an orientation was resolved by Kamada and Oshiro \cite{MR2371714, MR2657689}, who introduced symmetric quandles and their homology groups.  In these works, they defined a quandle cocycle invariant $\Phi_{\theta}$ for an unoriented link with respect to a given symmetric quandle 2-cocycle $\theta$. They further showed that $\Phi_{\theta}$ agrees with the quandle cocycle invariant $\Phi^{\mathrm{ori}}_{\theta}$ of an oriented link when the orientation is  arbitrarily assigned. Consequently, when the quandle 2-cocycle $\theta$ is cohomologous to a symmetric quandle 2-cocycle, the invariant $\Phi^{\mathrm{ori}}_{\theta}$ does not depend on the orientation of the link.  An analogous invariant has been defined for unoriented surface links via symmetric quandle 3-cocycles, and similar results have ensued. One of the remarkable consequences of this is enabling us to estimate the minimal triple point numbers of non-orientable surface links.
\par

Using certain symmetric quandles called 4-fold symmetric quandles, an invariant of 3-manifolds has been defined in \cite{MR2945646, MR2821435}. Furthermore, it has been shown that the Chern--Simons invariant of closed 3-manifolds can be regarded as a quandle cocycle invariant. In \cite{MR2890467}, symmetric quandles have been used to define invariants for spatial graphs and handlebody-links. The aforementioned coupled with further modifications (see, for example \cite{MR4476068, MR2629767, MR4594919}) have justified the relevance of symmetric quandles to low dimensional topology.
\par

Inspired by \cite{MR1994219, MR2155522}, we explore these objects from a categorical perspective enabling us to take a different approach than the already existing ones. As a consequence, we are able to completely determine the Beck modules in the category of symmetric quandles. Introduced by Beck in \cite{MR2616383}, given a category $\mathbf{C}$ and an object $X$ in $\mathbf{C}$, a Beck module over $X$ is an abelian group object in the slice category $\mathbf{C}|_X.$ Beck modules provide a very general notion of a coefficient module for an associated (co)homology theory. To exemplify their strength, we note that they simultaneously generalise the notion of coefficient modules from group cohomology, Lie algebra cohomology as well as Hochschild cohomology of associative algebras. Once the Beck modules in a category have been identified, it is natural to ponder over a corresponding (co)homology theory. We effectively formulate such a (co)homology theory for symmetric quandles with homogeneous coefficients. We also explore the related extension theory and establish connections with low dimensional group cohomology of associated groups. Similar constructions and results work in the category of racks as well.
\par 

 It is noteworthy to mention that this generalized (co)homology theory for symmetric quandles with homogeneous coefficients can be conceived in a manner reminiscent of group and Hochschild cohomology in which the coefficient objects admit actions. The significance of this (co)homology theory lies in its potential to yield invariants (for example, for spatial graphs and low dimensional manifolds) that are more comprehensive and capable of capturing additional information compared to the original (co)homology, as initially defined by Kamada and Oshiro \cite{MR2657689}, and which is a specific case where the action is trivial.

\par
The paper is organised as follows. In Section \ref{section prelim}, we review some basic definitions and key concepts from the theory of symmetric racks and symmetric quandles that will be used in the subsequent sections. Section \ref{modules over symmetric racks} introduces modules over symmetric racks and symmetric quandles as trunk maps from certain associated trunks to the category of abelian groups, along with some examples. In Section \ref{section category of modules}, we describe abelian group objects in the slice category over a fixed symmetric rack. This is achieved by first defining the semi-direct product of a symmetric rack and its module. Theorem \ref{equivalence of categories for symmetric rack modules} proves that the category of modules over a symmetric rack $(X,\rho_X)$ is equivalent to the category of abelian group objects in the slice category over $(X,\rho_X)$. An analogous result for symmetric quandles is deduced in Theorem \ref{equivalence of categories for symmetric quandle modules}. In Section \ref{section abelian extensions of symmetric racks},  we define extensions of symmetric racks and symmetric quandles. Theorem \ref{theorem ext cohom} establishes that the set of equivalence classes of extensions of a symmetric rack $(X,\rho_X)$ by an $(X,\rho_X)$-module $\mathscr{F}$ are in bijective correspondence with an abelian group $\mathcal{H}^2_{\mathrm{SR}}(X,\mathscr{F})$ defined via extensions. An analogous result for symmetric quandles is proved in Theorem \ref{theorem ext cohom quandle}. In Section \ref{section generalized cohomology}, after defining the suitable symmetric algebra, we introduce a generalized (co)homology theory for symmetric racks and symmetric quandles. This theory recovers the symmetric quandle (co)homology developed by Kamada and Oshiro \cite{MR2657689}. Finally, in Section \ref{relation cohom symm rack and group}, we establish an isomorphism between the second cohomology of a symmetric rack and the first cohomology of its associated group when the coefficients are trivial homogeneous modules.
\medskip


\section{Preliminaries}\label{section prelim}
In this section, we review some basic definitions and key concepts required for what follows.

\subsection{Symmetric racks and symmetric quandles}   
We begin with some basic ideas introduced in \cite{MR2657689}. 

\begin{definition}
Let $(X,*)$ be a rack (respectively quandle). A map $\rho_X:X \rightarrow X$ is called a \textit{good involution} if 
\begin{enumerate}
	\item[(S1)] $\rho_X$ is an involution, $\ie$, $\rho_X^2=\id_X$,
	\item[(S2)] $\rho_X(x \ast y)= \rho_X(x) \ast y$,
	\item[(S3)] $x \ast \rho_X(y) = x \ast^{-1}y$,
\end{enumerate}
for all $x, y \in X$.  The pair $(X, \rho_X)$ is called a \textit{symmetric rack (respectively symmetric quandle)}.
\end{definition}

We refer the reader to \cite{MR2371714,MR3363811, MR2657689,MR4594919} for more details on these structures and their applications in knot theory.

\begin{example}
The following are some elementary examples:
\begin{enumerate}
\item If $X$ is a trivial quandle, then every involution of $X$ is a good involution.
\item If $X= \Conj(G)$ is the conjugation quandle of a group $G$, then $\rho_X:X \rightarrow X$  given by $\rho_X(x)=x^{-1}$ is a good involution.
\item If $X= \Core(G)$ is the core quandle of a group $G$, then the identity map $\id:X \rightarrow X$ is a good involution. Also, if $z \in G$ is a central element of order two, then $\rho_X:X \rightarrow X$ given by $\rho(y)=y z$ is another good involution of $X$.
	\end{enumerate}
\end{example}	

\begin{definition}
	A \textit{symmetric rack (respectively quandle) homomorphism} is a map $f:(X,\rho_X) \rightarrow (Y, \rho_Y)$ such that 
	$$f(x \ast y)=f(x) \ast f(y) \quad \textrm{and} \quad f(\rho_X(x))=\rho_Y(f(x))$$ for all $x,y \in X$. A \textit{symmetric rack (respectively quandle) isomorphism} is a symmetric rack (respectively quandle) homomorphism which is also a bijection.
\end{definition}

Analogous to the associated group of a rack (respectively quandle), Kamada and Oshiro defined the associated group of a symmetric rack (respectively quandle) \cite[Section 4]{MR2657689}.

\begin{definition}
Let $(X, \rho_X)$ be a symmetric rack (respectively quandle). Then the group
	$$G_{(X,\rho_X)}=\langle e_x, x \in X \mid e_{x \ast y}=e_y^{-1}e_x e_y, ~e_{\rho_X(x)}=e_x^{-1} \text{ for }x,y \in X \rangle$$
is called the {\it associated group} of $(X, \rho_X)$.
\end{definition}

\begin{example} \cite[Example 2.5]{MR2657689}
Let $(X, \rho_X)$ be the fundamental symmetric quandle of an unknot in $\mathbb{R}^3$ or an unknotted 2-sphere in $\mathbb{R}^4$. To be precise, $X = \{x_1, x_2\}$ and $\rho_X(x_1) = x_2$.  Then the associated group $G_{(X,\rho_X)}  \cong \mathbb{Z}$.
\end{example}
\medskip

Let  $\SR$ and $\SQ$ denote the categories whose constituent objects are the symmetric racks and symmetric quandles, respectively, with the morphisms as defined above. It is clear that finite direct products exist in both $\SR$ and $\SQ$. Further, equalisers exist in both $\SR$ and $\SQ$. These are just the set-theoretical equalisers with the induced $\ast$ and $\rho_X$. Since, both $\SR$ and $\SQ$  have all binary products and equalisers, they admit all finite limits. \cite[Proposition 2.8.2]{MR1291599}.
\par 

Let $(X,\rho_X)$ be a  fixed symmetric rack and $\SR|_{(X,\rho_X)}$ denote the slice category over $(X,\rho_X)$. More precisely, the objects in $\SR|_{(X,\rho_X)}$ are the symmetric rack homomorphisms $f:(Y,\rho_Y) \rightarrow (X,\rho_X)$, and a morphism from $f:(Y,\rho_Y) \rightarrow (X,\rho_X)$ to $g:(Z,\rho_Z) \rightarrow (X,\rho_X)$ is a symmetric rack homomorphism $h:(Y,\rho_Y) \rightarrow (Z,\rho_Z)$ such that $g \,h=f$. The slice category over a symmetric quandle is defined analogously.
\par 

Since $\SR$ and $\SQ$ have all finite limits, it follows that $\SR|_{(X,\rho_X)}$ and $\SQ|_{(X,\rho_X)}$ also have finite limits \cite[using Proposition 2.8.2]{MR1291599}.
In particular, the product of $f:(Y,\rho_Y) \rightarrow (X,\rho_X)$ and $g:(Z,\rho_Z) \rightarrow (X,\rho_X)$ in $\SR|_{(X,\rho_X)}$ is the pullback $$h:(Y \times_X Z, \,\rho_{Y \times_X Z}) \rightarrow (X,\rho_X)$$ of $f$ and $g$, where $Y \times_X Z=\{(y,z) \in Y \times Z \mid f(y)=g(z)\}$ has the usual product rack structure, $\rho_{Y \times_X Z}(y,z)=(\rho_Y(y),\rho_Z(z))$ and $h(y, z)= f(y)=g(z)$ for all $(y,z) \in Y \times_X Z$.  The same construction works in $\SQ$. Thus, we can consider group objects in $\SR|_{(X,\rho_X)}$ and $\SQ|_{(X,\rho_X)}$. 
\medskip

\subsection{Trunks}
Trunks were first introduced and studied by Fenn, Rourke, and Sanderson in \cite{MR1364012}. A \textit{trunk} $\T$ is an object analogous to a category which consists of a class of objects and a set $\Hom_{\T}(A,B)$ of morphisms along with some commutative squares
\begin{center}
	\begin{tikzcd}[column sep=3em]
		A \arrow[r, "f"] \arrow[d, "g"] & B \arrow[d, "h"] \\
		C \arrow[r, "i"] & D,
	\end{tikzcd}
\end{center}
called {\it preferred squares}. Given two trunks $\Sa$ and $\T$, a \textit{trunk map} $F:\Sa \rightarrow \T$ is a map which assigns to each object $A$ of $\Sa$, an object $F(A)$ of $\T$, and to every morphism $f:A \rightarrow B$ of $\Sa$, a morphism $F(f):F(A) \rightarrow F(B)$ of $\T$ such that a preferred square
\begin{center}
	\begin{tikzcd}[column sep=3em]
		A \arrow[r, "f"] \arrow[d, "g"] & B \arrow[d, "h"] \\
		C \arrow[r, "i"] & D,
	\end{tikzcd}
\end{center}
is mapped to a preferred square
\begin{center}
	\begin{tikzcd}[column sep=3em]
		F(A) \arrow[r, "F(f)"] \arrow[d, "F(g)"] & F(B) \arrow[d, "F(h)"] \\
		F(C) \arrow[r, "F(i)"] & F(D).
	\end{tikzcd}
\end{center}

For example, for any category $\C$, we have a well-defined trunk $T(\C)$, which has the same objects and morphisms as $\C$, and whose preferred squares are the commutative diagrams in $\C$. In particular, we denote the trunk associated to the category of abelian groups by $\Ab$.	
\medskip


\section{Modules over symmetric racks and symmetric quandles}\label{modules over symmetric racks}
Given a symmetric rack (respectively quandle) $(X,\rho_X)$, we define a trunk $\T(X, \rho_X)$ as follows: 

\begin{itemize}
	\item  for each $x \in X$, $\T(X, \rho_X)$ has precisely one object,
	\item for each $x \in X$, there are morphisms $\gamma_x:x \rightarrow \rho_X(x)$ and $\id_x: x \rightarrow x$,
	\item for each ordered pair $(x,y)$ of elements of $X$, there are morphisms $\alpha_{x,y}:x \rightarrow x \ast y$ and $\beta_{x,y}:y \rightarrow x *y$ such that the following squares are preferred squares for all $x,y,z \in X$:
	\begin{center}
		\begin{tikzcd}[row sep=3em, column sep=4em]
			x \arrow[r, "\alpha_{x,y}", shorten >=2pt, shorten <=2pt] \arrow[d, "\alpha_{x,z}", shorten >=0pt, shorten <=0pt] & x \ast y \arrow[d, "\alpha_{x \ast y,z}", shorten >=0pt, shorten <=0pt] \\
			x \ast z \arrow[r, "\alpha_{x\ast z, y \ast z}"'] & (x\ast y) \ast z
		\end{tikzcd}
		\hspace{2cm}
		\begin{tikzcd}[row sep=3em, column sep=4em]
			y \arrow[r, "\beta_{x,y}", shorten >=2pt, shorten <=2pt] \arrow[d, "\alpha_{y,z}"', shorten >=0pt, shorten <=0pt] & x \ast y \arrow[d, "\alpha_{x \ast y,z}", shorten >=0pt, shorten <=0pt] \\
			y \ast z \arrow[r,"\beta_{x\ast z, y \ast z}"',shorten >=-4pt, shorten <=-4pt] & 
			(x \ast y)\ast z
		\end{tikzcd}
	\end{center}
	
	\begin{center}
		\begin{tikzcd}[row sep=3em, column sep=5em]
			x \arrow[r, "\alpha_{x,y}", shorten >=2pt, shorten <=2pt] \arrow[d, "\gamma_x", shorten >=0pt, shorten <=0pt] & x \ast y \arrow[d, "\gamma_{x \ast y}", shorten >=0pt, shorten <=0pt] \\
			\rho_X(x) \arrow[r, "\alpha_{\rho_X(x), y}"'] & 
			\rho_X(x) \ast y
		\end{tikzcd}
		\hspace{2cm}
		\begin{tikzcd}[row sep=3em, column sep=5em]
			x \arrow[r, "\beta_{x,y}", shorten >=2pt, shorten <=2pt] \arrow[d, "\id_x"', shorten >=0pt, shorten <=0pt] & x \ast y \arrow[d, "\gamma_{x \ast y}", shorten >=0pt, shorten <=0pt] \\
			\rho_X(x) \arrow[r, "\beta_{\rho_X(x), y}"'] & 
			\rho_X(x) \ast y
		\end{tikzcd}
	\end{center}
	
	\begin{center}
		\hspace{0.8cm}
		\begin{tikzcd}[row sep=3em, column sep=6em]
			x \arrow[r, "\gamma_x", shorten >=2pt, shorten <=2pt] \arrow[d, "\id_x", shorten >=0pt, shorten <=0pt] & \rho_X(x) \arrow[d, "\gamma_{\rho_X(x)}", shorten >=0pt, shorten <=0pt] \\
			x \arrow[r, "\id_x"'] & 
			x
		\end{tikzcd}
		\hspace{2.2cm}
		\begin{tikzcd}[row sep=3em, column sep=5.5em]
			x \arrow[r, "\alpha_{x,\rho_X(y)}", shorten >=0pt, shorten <=0pt] \arrow[d, "\id_x"', shorten >=0pt, shorten <=0pt] & x \ast \rho_X(y) \arrow[d, "\alpha_{x \ast \rho_X(y),y}", shorten >=0pt, shorten <=0pt] \\
			x \arrow[r, "\id_x"'] & 
			x
		\end{tikzcd}
	\end{center}
\end{itemize}

Thus, a trunk map $A:\T(X, \rho_X) \rightarrow \Ab$ yields abelian groups $A_x$ and group homomorphisms $\phi_{x,y}:A_x \rightarrow A_{x \ast y}$, $\psi_{x,y}:A_y \rightarrow A_{x \ast y}$ and $\eta_x:A_x \rightarrow A_{\rho_X(x)}$ such that
\begin{enumerate}[(M1)]
	\item \label{M1} $\phi_{x \ast y,z}\phi_{x,y}=\phi_{x \ast z,y \ast z}\phi_{x,z}$,
	\item \label{M2} $\phi_{x \ast y,z}\psi_{x,y}=\psi_{x \ast z,y \ast z}\phi_{y,z}$,
	\item \label{M3} $\eta_{\rho_X(x)} \eta_x = \id$,
	\item \label{M4} $\eta_{x \ast y} \phi_{x,y}= \phi_{\rho_X(x),y}\eta_x$,
	\item \label{M5} $\psi_{\rho_X(x),y}=\eta_{x \ast y}\psi_{x,y}$,
	\item \label{M6}$\phi_{x \ast^{-1}y,y}\phi_{x, \rho_X(y)}=\id$,
\end{enumerate}
 for all $x,y,z \in X$. We denote such a trunk map by $\mathscr{F}=(A,\phi,\psi,\eta)$.

\begin{definition}
	Let $(X,\rho_X)$ be a symmetric rack. Then an {\it $(X,\rho_X)$-module} is a trunk map $\mathscr{F}=(A,\phi,\psi,\eta): T(X,\rho_X)\rightarrow \Ab$ such that $\phi_{x,y}: A_{x}\rightarrow A_{x\ast y}$ is an isomorphism and 
	\begin{enumerate}
		\item[(M7)] \label{M7} $\psi_{x\ast y,z}(a)= \phi_{x\ast z,y\ast z}\psi_{x,z}(a)+\psi_{x\ast z,y\ast z}\psi_{y,z}(a)$,
		\item[(M8)] \label{M8} $\phi_{x\ast^{-1} y,y}\psi_{x,\rho_X(y)}(\eta_{y}(b))=-\psi_{x\ast \rho_X(y),y}(b)$,
	\end{enumerate}
	hold for all $a \in A_z$, $b \in A_y$ and $x,y, z \in X$. 
	\par
	
	If $(X,\rho_X)$ is a symmetric quandle, then we desire that the trunk map $\mathscr{F}$ additionally satisfies the condition
	\begin{enumerate}\label{M9}
		\item [(M9)] $\phi_{x,x}(a)+\psi_{x,x}(a)=a$
	\end{enumerate} 
	for all $a\in A_x$ and $x \in X.$
\end{definition}

\begin{definition}
Let $(X,\rho_X)$ be a symmetric rack (respectively quandle). An $(X,\rho_X)$-module  $\mathscr{F}=(A,\phi,\psi,\eta)$ is called \textit{homogeneous} if the constituent groups are all isomorphic, that is, $A_x \cong A_y$ for all $x,y \in X$.
\end{definition}		

Before proceeding further, let us consider some examples of modules.

\begin{example}
Let $(X,\rho_X)$ be a symmetric rack or a symmetric quandle. 
\begin{enumerate}
\item Any abelian group $A$ can be considered as an $(X,\rho_X)$-module by taking $A_x=A$ $\phi_{x,y}=\id_A$, $\psi_{x,y}=0$ and $\eta_x=-\id_A$ for all $x,y \in X$, which we call a \textit{trivial homogeneous module}.
\item \label{example-module} Consider a family of abelian groups $\{ G_{i}\}_{i\in I}$, where $I$ is an indexing set. Let $\alpha_i, \beta_i \in \Aut(G_i)$ such that  $\alpha_{i}^2=\id_{G_i}=\beta_{i}^2$ and $\alpha_{i}\beta_{i}=\beta_{i}\alpha_{i}$ for all $i \in I$. Then, $\mathscr{F}=(A,\phi,\psi,\eta)$ is an $(X,\rho_X)$-module, where
$A_x = \prod_{i\in I} G_i $, $\phi_{x,y}= \prod_{i \in I} \alpha_i$,  $\psi_{x,y} = 0$ and $\eta_{x}=\prod_{i \in I} \beta_i$ for all $x \in X$.
\item  Let $A$ an abelian group whose each element has order 2 or 4. Then, $\mathscr{F}=(A,\phi,\psi,\eta)$ is an $(X,\rho_X)$-module, where $A_x = A$, $\phi_{x,y}(a)=-a$, $\psi_{x,y}(a) = 2a$ and $\eta_{x}(a)=-a$ for all $x \in X$ and $a \in A$.
\item By \cite[Section 4]{MR2657689}, we have a well-defined right action of $G_{(X,\rho_X)}$ on $(X,\rho_X)$, given by $$x \cdot e_y= x*y$$ for all $x,y \in G_{(X,\rho_X)}$.
	 Let $ [x]$ denote the orbit of $x$ under this action.  Then, we have an $(X,\rho_X)$-module $\mathscr{F}=(A,\phi,\psi,\eta)$ with $A_x=A_{\rho_X(x)}:=A_{[x]}$ and 
	\begin{equation*}
		\phi_{x,y}: A_{[x]} \rightarrow A_{[x\ast y]} \text{ defined by }
		a \mapsto f_{[x]}(a),
	\end{equation*}
	\begin{equation*}
		\psi_{x,y}: A_{[x]} \rightarrow A_{[x\ast y]} \text{ defined by }
		b \mapsto 0 \text{ and }
	\end{equation*}  
	\begin{equation*}
		\eta_x: A_{[x]} \rightarrow A_{[\rho_X(x)]} \text{ defined by }
		c \mapsto g_{[x]}(c),
	\end{equation*}
	where $f_{[x]}, g_{[x]}: A_{[x]} \to A_{[x]}$ are group homomorphisms such that  $f_{[x]}^2=\id=g_{[x]}^2$ and $f_{[x]}g_{[x]}=g_{[x]}f_{[x]}$ for all $x \in X.$

\end{enumerate}
\end{example}

\begin{definition}\label{map}
Let $(X,\rho_X)$ be a symmetric rack (respectively quandle), and $\mathscr{F}=(A,\phi,\psi,\eta)$ and $\mathscr{F'}=(A',\phi',\psi',\eta')$ be $(X,\rho_X)$-modules. An \textit{$(X,\rho_X)$-map} is a natural transformation $f: \mathscr{F}\rightarrow \mathscr{F}'$ of trunk maps, that is, it is a collection  $f=\{f_x: A_x\rightarrow A'_x \mid  x \in X\}$ of group homomorphisms such that
	\begin{eqnarray}\label{X-map}
		\phi'_{x,y}f_x &=& f_{x\ast y} \phi_{x,y},\\
		\psi'_{x,y}f_y &=& f_{x\ast y}\psi_{x,y},\\
		\eta'_{x}f_x &=& f_{\rho_X(x)} \eta_{x},
	\end{eqnarray}
	for all $x,y \in X$.
	\par
	In addition, if each $f_x$ is an isomorphism of groups, then we say that $f: \mathscr{F}\rightarrow \mathscr{F}'$  is an isomorphism of  $(X,\rho_X)$-modules. 
\end{definition}

\begin{remark}
	If $(X, \rho_X)$ is a symmetric rack, then we can form the category $\mathbf{SRMod}_{(X, \rho_X)}$, whose objects are $(X,\rho_X)$-modules and whose morphisms are $(X,\rho_X)$-maps. Similarly, if $(X, \rho_X)$ is a symmetric quandle, then we can form the category $\mathbf{SQMod}_{(X, \rho_X)}$ of modules over $(X,\rho_X)$. Moreover, both $\mathbf{SRMod}_{(X, \rho_X)}$ and $\mathbf{SQMod}_{(X, \rho_X)}$) are abelian categories.
\end{remark}
\medskip


\section{Category of modules over symmetric racks and symmetric quandles}\label{section category of modules}
In this section, we describe the objects in the category $\Ab(\SR|_{(X,\rho_X)})$ of abelian group objects in the slice category $\SR|_{(X,\rho_X)}$ over a fixed symmetric rack $(X,\rho_X)$. Let $(X,\rho_X)$ be a symmetric rack and $\mathscr{F}=(A,\phi,\psi,\eta)$ an $(X,\rho_X)$-module. We define the \textit{semi-direct product} $\mathscr{F} \rtimes X$ of $\mathscr{F}$ and $X$ to be the set
$$ \big\{(a,x) \mid x\in X , a \in A_x \big\}$$ 
equipped with the binary operation
\begin{equation}
	(a,x) \tilde{\ast} (b,y):=\big(\phi_{x,y}(a)+\psi_{x,y}(b), x \ast y \big)
\end{equation}
and the map $\rho_{\mathscr{F} \rtimes X}:\mathscr{F} \rtimes X \rightarrow \mathscr{F} \rtimes X$ defined by 
$$\rho_{\mathscr{F} \rtimes X}\big((a,x) \big)= \big(\eta_x(a), \rho_X(x)\big)$$
for all $a \in A_x$,  $b \in A_y$ and $x,y \in X$.

\begin{proposition}\label{semi-direct product}
	Let $(X,\rho_X)$ be a symmetric rack and $\mathscr{F}=(A,\phi,\psi,\eta)$ an $(X,\rho_X)$-module. Then the semi-direct product $({\mathscr{F} \rtimes X},\rho_{\mathscr{F} \rtimes X})$ is a symmetric rack. Furthermore, if  $(X,\rho_X)$ is a symmetric quandle, then so is $({\mathscr{F} \rtimes X},\rho_{\mathscr{F} \rtimes X})$.
\end{proposition}

\begin{proof}
	Note that ${\mathscr{F} \rtimes X}$ is a rack, as proved in \cite[Proposition 2.1]{MR2155522}. It remains to check that $\rho_{\mathscr{F} \rtimes X}$ is a good involution.
	\begin{itemize}
		\item For (S1), we see that $\rho_{\mathscr{F} \rtimes X}^2(a,x)= \rho_{\mathscr{F} \rtimes X}(\eta_x(a), \rho_X(x))=(\eta_{\rho_X(x)}(\eta_x(a)), \rho_X^2(x))=(a,x)$, where the last equality follows from the $(X,\rho_X)$-module axiom \hyperref[M3]{(M3)}.
		\item For (S2), we have
		\begin{eqnarray*}
			\rho_{\mathscr{F} \rtimes X}\big((a,x)\big) \tilde{\ast} (b,y)&=&(\eta_x(a), \rho_X(x)) \tilde{\ast}(b,y)\\
			&=&(\phi_{\rho_X(x),y}(\eta_x(a))+\psi_{\rho_X(x),y}(b), \rho_X(x) \ast y)\\
			&=& (\eta_{x \ast y}(\phi_{x,y}(a)+\psi_{x,y}(b)), \rho_X(x \ast y)),\\
			&& \textrm{by $(X, \rho_X)$-module axioms \hyperref[M4]{(M4)} and \hyperref[M5]{(M5)}, and $\rho_X$ being a good involution}\\
			&=& \rho_{\mathscr{F} \rtimes X}(\phi_{x,y}(a)+\psi_{x,y}(b), x \ast y)\\
			&=& \rho_{\mathscr{F} \rtimes X} \big((a,x) \tilde{\ast}(b,y) \big).
		\end{eqnarray*}
		\item For (S3), we have 
		\begin{eqnarray*}
			(a,x) \tilde{\ast} \rho_{\mathscr{F} \rtimes X} \big((b,y)\big) &=& (a,x) \tilde{\ast} (\eta_y(b), \rho_X(y))\\
			&=& (\phi_{x, \rho_X(y)}(a)+\psi_{x,\rho_X(y)}(\eta_y(b)), x \ast \rho_X(y)),\\
			&& \textrm{by $(X, \rho_X)$-module axioms \hyperref[M6]{(M6)} and \hyperref[M8]{(M8)}, and $\rho_X$ being a good involution}\\
			&=& (\phi_{x \ast^{-1}y,y}^{-1}(a-\psi_{x \ast^{-1}y,y}(b)), x \ast^{-1}y)\\
			&=& (a,x)\tilde{\ast}^{-1}(b,y).
		\end{eqnarray*}
	\end{itemize}
	Thus, $\rho_{\mathscr{F} \rtimes X}$ is a good involution, and hence $({\mathscr{F} \rtimes X},\rho_{\mathscr{F} \rtimes X})$ is a symmetric rack. Finally, suppose that $(X,\rho_X)$ is a symmetric quandle. Then, we have the additional condition $\phi_{x,x}(a)+\psi_{x,x}(a)=a$ for all $x\in X$ and $a\in A_x$. Thus, we obtain
$$(a,x) \tilde{\ast} (a,x)=\big(\phi_{x,x}(a)+\psi_{x,x}(a), x \big)= (a, x)$$
for all $x\in X$ and $a\in A_x$, and hence $({\mathscr{F} \rtimes X},\rho_{\mathscr{F} \rtimes X})$ is a symmetric quandle.
\end{proof}

\begin{proposition}\label{abelian-group-object}
Let $(X, \rho_X)$ be a symmetric rack and $\mathscr{F}=(A,\phi,\psi,\eta)$ an  $(X, \rho_X)$-module. Then there exists an abelian group object $p:({\mathscr{F} \rtimes X},\rho_{\mathscr{F} \rtimes X}) \rightarrow (X, \rho_X)$, which we denote by $\mathcal{T}(\mathscr{F})$, in the slice category over $(X, \rho_X)$. Further, let  $\mathscr{F'}=(A',\phi',\psi',\eta')$ be another $(X,\rho_X)$-module and $f:\mathscr{F}\rightarrow \mathscr{F'}$ an $(X, \rho_X)$-map. Then the map $\mathcal{T}(f): \mathscr{F}\rtimes X \rightarrow \mathscr{F'}\rtimes X$ defined by $\mathcal{T}(f)\big((a,x) \big)= (f_{x}(a),x)$ gives a slice morphism. Moreover, $\mathcal{T}: \mathbf{SRMod}_{(X, \rho_X)} \rightarrow \Ab(\SR|_{(X,\rho_X)})$ is a functor.
\end{proposition}

\begin{proof}
Let $\mathscr{F}=(A,\phi,\psi,\eta)$ be an $(X,\rho_X)$-module. Then, we have the semi-direct product $({\mathscr{F} \rtimes X},\rho_{\mathscr{F} \rtimes X})$, which is a symmetric rack due to Proposition \ref{semi-direct product}. Let $p: \mathscr{F} \rtimes X \rightarrow X$ be the natural projection given as $p(a,x)= x$ for $a \in A_x$ and $x \in X$. Then $p$ is a homomorphism of symmetric racks. We claim that $p: \mathscr{F} \rtimes X \rightarrow X$ has the canonical structure of an abelian group object. Let $$\big((\mathscr{F} \rtimes X )\times_X (\mathscr{F} \rtimes X),\rho_{(\mathscr{F} \rtimes X )\times_X (\mathscr{F} \rtimes X)}\big) \rightarrow (X, \rho_X)$$ be the categorical product of $p: \mathscr{F} \rtimes X \rightarrow X$ with itself in the slice category over $(X,\rho_X)$. The symmetric rack structure is defined in the same way as in the cartesian product. We define the morphisms
 \begin{itemize}
\item $\mu:(\mathscr{F} \rtimes X )\times_X (\mathscr{F} \rtimes X) \rightarrow (\mathscr{F} \rtimes X) $ given by $((a_1,x),(a_2,x)) \mapsto (a_1+a_2,x)$,
\item $\sigma:X \rightarrow (\mathscr{F} \rtimes X)$ given by $x \mapsto (0,x)$,
\item $\nu:(\mathscr{F} \rtimes X) \rightarrow (\mathscr{F} \rtimes X)$ given by $(a,x) \mapsto (-a,x)$,
\end{itemize}
for all $a_1,a_2, a \in A_x$ and $x \in X$. It easy to check that $\mu$, $\sigma$ and $\nu$ are symmetric rack homomorphisms. Also, the fact that $\mu$, $\sigma$ and $\nu$ are slice morphisms and that they turn $p: \mathscr{F} \rtimes X \rightarrow X$ into an abelian group object follows directly (see also \cite[Theorem 2.2]{MR2155522}).
\par	

Let us define a functor $\mathcal{T}: \mathbf{SRMod}_{(X, \rho_X)} \rightarrow \Ab(\SR|_{(X,\rho_X)})$.  Given an $(X,\rho_X)$-module $\mathscr{F}=(A,\phi,\psi,\eta)$, let $\mathcal{T}(\mathscr{F})$ denote the object $p:\mathscr{F}\rtimes X\rightarrow X$. Let $\mathscr{F'}=(A',\phi',\psi',\eta')$ be another $(X,\rho_X)$-module and $f:\mathscr{F}\rightarrow \mathscr{F'}$ an $(X, \rho_X)$-map. Define $\mathcal{T}(f):\mathscr{F}\rtimes X \rightarrow \mathscr{F'}\rtimes X$ by $$T(f)(a,x)=(f_{x}(a),x).$$ Then, we have
			\begin{eqnarray*}
				\mathcal{T}(f) \big((a,x) \star (b,y) \big)&=&\mathcal{T}(f)\big(\phi_{x,y}(a)+\psi_{x,y}(b),x \ast y \big)\\
				&=& \big(f_{x \ast y}(\phi_{x,y}(a)+\psi_{x,y}(b)) ,x \ast y \big)\\
				&=&\big(\phi'_{x,y}f_x(a)+\psi'_{x,y}f_y(b),x \ast y \big), \quad \textrm{since $f$ is an $(X, \rho_X)$-map}\\
				&=& \mathcal{T}(f)\big((a,x)\big) \star' \mathcal{T}(f)\big((b,y)\big),
			\end{eqnarray*}
where $\star$ is the rack operation in $\mathscr{F}\rtimes X$ and $\star'$ is the rack operation in $\mathscr{F'}\rtimes X$. Further, since $\eta'_{x}f_x = f_{\rho_X(x)} \eta_{x}$, the  diagram
\begin{center}
\begin{tikzcd}[row sep=3em, column sep=4em]
\mathscr{F}\rtimes X \arrow[r, "\rho_{\mathscr{F}\rtimes X}"] \arrow[d, "\mathcal{T}(f)"] & \mathscr{F}\rtimes X \arrow[d, "\mathcal{T}(f)"] \\
\mathscr{F'}\rtimes X \arrow[r, "\rho_{\mathscr{F'}\rtimes X}"] & \mathscr{F'}\rtimes X
\end{tikzcd}
\end{center}
 commutes. Thus,  $T(f)$ is a homomorphism of symmetric racks. Finally, since the diagram
\begin{center}
\begin{tikzcd}
\mathscr{F}\rtimes X \arrow[r, "\mathcal{T}(f)"] \arrow[rd, "p"'] & \mathscr{F'}\rtimes X \arrow[d, "p'"] \\
& X
\end{tikzcd}
\end{center}
commutes, where $p'$ is the projection to the second coordinate, $\mathcal{T}(f)$ is a slice morphism. In fact, one can check that $\mathcal{T}$ is a functor (see \cite[Theorem 2.2]{MR2155522}).	
		\end{proof}
	
\begin{proposition}\label{module}
Let $(X, \rho_X)$ be a symmetric rack and $p:(Y,\rho_Y) \rightarrow (X, \rho_X)$ an abelian group object with the multiplication map $m$, the inverse map $i$, and the section $s$. Let $(X, *)$ and $(Y, \star)$ be the underlying racks. Then there exists an $(X,\rho_X)$-module $(R, \phi, \psi, \eta)$ given by an induced abelian group structure on the fibre $R_x:=p^{-1}(x)$ for each $x \in X$, where
		$$\phi_{x,y}:R_x \rightarrow R_{x \ast y} ~~ \text{is given by }~ u \mapsto u \star s(y),$$  
		$$\psi_{x,y}:R_y \rightarrow R_{x \ast y} ~~ \text{is given by }~ v \mapsto s(x) \star v$$ and 
		$$\eta_x: R_x \rightarrow R_{\rho_X(x)} ~~ \text{is given by }~ u \mapsto \rho_{Y}(u)$$
		for all $u\in R_x$, $v \in R_y$ and $x,y \in X$. Furthermore, the association gives a functor $\mathcal{S}: \Ab(\SR|_{(X,\rho_X)}) \to  \mathbf{SRMod}_{(X, \rho_X)}$. 
	\end{proposition}

\begin{proof}
Let $p:(Y,\rho_Y) \rightarrow (X, \rho_X)$ be an abelian group object in $\Ab(\SR|_{(X,\rho_X)})$. Let $R_x=p^{-1}(x)$ for each $x \in X$. Then an abelian group structure on $R_x$ is given by $$u+v:=m(u,v),$$ where $s(x)$ is the identity element in $R_x$ and $-u:=i(u)$ for all $u,v \in R_x$. To define an $(X, \rho_X)$-module structure, let us define
	$$\phi_{x,y}:R_x \rightarrow R_{x \ast y} ~~\text{given by } ~u \mapsto u \star s(y),$$  
	$$\psi_{x,y}:R_y \rightarrow R_{x \ast y} ~~\text{given by }~ v \mapsto s(x) \star v$$ and 
	$$\eta_x: R_x \rightarrow R_{\rho_X(x)} ~~\text{given by }~ u \mapsto \rho_{Y}(u)$$
	for all $u\in R_x$, $v \in R_y$ and $x,y \in X$.
It has been proved in  \cite[Theorem 2.2]{MR2155522} that $\phi_{x,y}$ and $\psi_{x,y}$ are group homomorphisms for all $x, y \in X$. For the case of $\eta_x$, if $u_1,u_2 \in R_x$, then we see that
	\begin{eqnarray*}
		\eta_x(u_1+u_2)&=& \rho_Y(m(u_1,u_2))\\
		&=& m(\rho_Y(u_1), \rho_Y(u_2)), \quad \textrm{since $m \,(\rho_Y \times \rho_Y)=\rho_Y \, m$}\\
		&=& \rho_Y(u_1)+\rho_Y(u_2)\\
		&=& \eta_x(u_1)+\eta_x(u_2).
	\end{eqnarray*}
Next, we need to check that $\phi_{x,y}$, $\psi_{x,y}$ and $\eta_x$ satisfy the $(X,\rho_X)$-module identities. Since \hyperref[M1]{(M1)}, \hyperref[M2]{(M2)} and \hyperref[M7]{(M7)} are proved in \cite[Theorem 2.2]{MR2155522}, we verify only the remaining identities.
\begin{itemize}
	\item Since $\rho_Y$ is a good involution, we have $$\eta_{\rho_X(x)} \eta_x(u)=\eta_{\rho_X(x)}(\rho_Y(u))=\rho_Y^2(u)=u,$$ and hence \hyperref[M3]{(M3)} holds.
	\item Again, since $\rho_Y$ is a good involution, we have $$\phi_{\rho_X(x),y}(\eta_x(u))=\phi_{\rho_X(x),y}(\rho_Y(u))=\rho_Y(u)\star s(y)=\rho_Y(u \star s(y))=\eta_{x \ast y}(u \star s(y))=\eta_{x \ast y} (\phi_{x,y}(u)),$$ and hence \hyperref[M4]{(M4)} holds.
	\item  Since $s$ is a homomorphism of symmetric racks, we see that $$\psi_{\rho_X(x),y}(v)= s(\rho_X(x))\star v=\rho_Y(s(x)) \star v=\rho_Y(s(x)\star v)=\eta_{x \ast y}(s(x) \star (v))=\eta_{x \ast y}(\psi_{x,y}(v)),$$ and hence \hyperref[M5]{(M5)} holds.
	\item Since $\rho_Y$ is a good involution and $s$ is a homomorphism of symmetric racks, we obtain
$$\phi_{x \ast^{-1}y,y}\phi_{x, \rho_X(y)}(u)=\phi_{x \ast^{-1},y}(u \star s(\rho_X(y)))=\phi_{x \ast^{-1},y}(u \star \rho_Y(s(y)))=(u \star \rho_Y(s(y)))\star s(y)=u,$$ and hence \hyperref[M6]{(M6)} holds.
	\item  Note that $\phi_{x\ast^{-1} y,y}\psi_{x,\rho_X(y)}(\eta_{y}(v))=\phi_{x\ast^{-1} y,y}\psi_{x,\rho_X(y)}(\rho_Y(v))=\phi_{x\ast^{-1} y,y}(s(x) \star \rho_Y(v))=(s(x) \star \rho_Y(v) )\star s(y)$ and
	$\psi_{x\ast \rho_X(y),y}(v)=s(x \ast \rho_X(y))\star v=(s(x) \star \rho_Y(s(y))) \star v$. Then, we have
	\begin{eqnarray*}
&& \phi_{x\ast^{-1} y,y}\psi_{x,\rho_X(y)}(\eta_{y}(v))+\psi_{x\ast \rho_X(y),y}(v)\\
&=& (s(x) \star \rho_Y(v)) \star s(y)+(s(x) \star \rho_Y(s(y))) \star v\\
&=&m \big((s(x) \star \rho_Y(v)) \star s(y), ~(s(x) \star \rho_Y(s(y))) \star v \big)\\
&=&m \big( \big((s(x), s(x)) \star (\rho_Y(v), \rho_Y(s(y))) \big) \star (s(y), v) \big), \quad \textrm{by definition of rack operation in $Y \times_X Y$}\\
&=& \big(m(s(x),s(x)) \star m \,(\rho_Y(v),\rho_Y(s(y))) \big)\star m(s(y),v), \quad \textrm{since $m$ is a homomorphism of racks}\\
&=& \big(s(x) \star \rho_Y \, m(v, s(y))) \star m(s(y),v), \quad \textrm{since $m$ is a homomorphism of symmetric racks}\\
&=& \big(s(x) \star \rho_Y \, m(s(y),v) )\star m(s(y),v), \quad \textrm{since $\rho_Y$ is a good involution}\\
&=& s(x),
	\end{eqnarray*}
\end{itemize}
and hence \hyperref[M8]{(M8)} holds. Hence, we have proved that $(R, \phi, \psi, \eta)$ is an $(X,\rho_X)$-module.
\par 

Next, we construct a functor $\mathcal{S}: \Ab(\SR|_{(X,\rho_X)}) \to  \mathbf{SRMod}_{(X, \rho_X)}$. Given an abelian group object $p:(Y,\rho_Y) \rightarrow (X, \rho_X)$, we set $\mathcal{S}(p) = (R,\phi,\psi,\eta)$, which is defined above. Let $f$ be a morphism from the abelian group object $p_1:(Y_1,\rho_{Y_1}) \rightarrow (X,\rho_X)$ to the abelian group object $p_2:(Y_2,\rho_{Y_2}) \rightarrow (X,\rho_X)$. Thus, if $m$ and $m'$ denote the corresponding multiplication maps, then $f \, m= m' \,(f, f)$. Further, $f:(Y_1,\rho_{Y_1}) \rightarrow (Y_2,\rho_{Y_2})$ is a homomorphism of symmetric racks with $p_2\,f=p_1$. Let $\mathscr{F}_1$ and $\mathscr{F}_2$ denote the $(X,\rho_X)$-modules as described above. Define $g:\mathscr{F}_1=(R_1,\phi,\psi,\eta) \rightarrow \mathscr{F}_2=(R_2,\phi',\psi',\eta')$ by setting
$$g_x:{(R_1)}_x \rightarrow {(R_2)}_x~~\text{as}~~g_x(u)= f(u)$$
for all $x\in X$ and $u \in {(R_1)}_x$. Since   $f \, m= m' \,(f, f)$, it follows that $g_x$ is an abelian group homomorphism for each $x \in X$. The identities  
$\phi'_{x,y}g_x = g_{x\ast y} \phi_{x,y}$ and  $\psi'_{x,y}g_y = g_{x\ast y}\psi_{x,y}$ are  proved in \cite[Theorem 2.2]{MR2155522}. Further, since
$$g_{\rho_X(x)} \eta_{x}(w)=g_{\rho_X(x)} (\rho_{Y_1}(w))=f(\rho_{Y_1}(w))=\rho_{Y_2}(f(w))=\rho_{Y_2}g_x(w)=\eta_x'(g_x(w))$$
for all $w\in (R_1)_x$, we have $\eta'_{x}g_x = g_{\rho_X(x)} \eta_{x}$. Hence, $g$ is an $(X,\rho_X)$-map. A routine check shows that $\mathcal{S}$ is a functor (see \cite[Thereom 2.2]{MR2155522}).
    \end{proof}

\begin{theorem}\label{equivalence of categories for symmetric rack modules}
Let $(X,\rho_X)$ be a symmetric rack. Then the category $\mathbf{SRMod}_{(X, \rho_X)}$ of $(X,\rho_X)$-modules is equivalent to the category $\Ab(\SR|_{(X,\rho_X)})$ of abelian group objects in the slice category over $(X,\rho_X)$.
\end{theorem}

\begin{proof} Let $(X,\rho_X)$ be a symmetric rack. By Proposition \ref{abelian-group-object} and Proposition \ref{module}, we have functors $\mathcal{T}:  \mathbf{SRMod}_{(X, \rho_X)} \to \Ab(\SR|_{(X,\rho_X)})$ and $\mathcal{S}: \Ab(\SR|_{(X,\rho_X)}) \to \mathbf{SRMod}_{(X, \rho_X)}$. To show that the categories $\mathbf{SRMod}_{(X, \rho_X)}$ and $ \Ab(\SR|_{(X,\rho_X)})$ are equivalent, we show that $\mathcal{T}\mathcal{S}$ is naturally isomorphic to $\id_{\Ab(\SR|_{(X,\rho_X)})}$ and $\mathcal{S}\mathcal{T}$ is naturally isomorphic to $\id_{ \mathbf{SRMod}_{(X, \rho_X)}}$.
\par

Let $\mathcal{O} \in  \Ab(\SR|_{(X,\rho_X)})$ denote the object $p:(Y,\rho_Y) \rightarrow (X, \rho_X)$ with the multiplication map $m$,  the inverse map $i$, and the section $s$. Then $\mathcal{T}\mathcal{S} (\mathcal{O})$ is the object $q:(\mathscr{R} \rtimes X, \rho_{\mathscr{R} \rtimes X}) \rightarrow (X, \rho_X)$ in the category $ \Ab(\SR|_{(X,\rho_X)})$, where $\mathscr{R}=(R, \phi,\psi,\eta)$ as constructed in Proposition \ref{module}. Let $*$ be the rack operation on $X$, and $\tilde{*}$ the rack operation on $\mathscr{R} \rtimes X$. Note that, the abelian group structure on $q:(\mathscr{R} \rtimes X, \rho_{\mathscr{R} \rtimes X}) \rightarrow (X, \rho_X)$ has the multiplication, the inverse and the section given by
\begin{itemize}
\item $\mu ((a_1,x),(a_2,x)) = (a_1+a_2,x)=(m(a_1,a_2),x)$,
\item $\sigma (x) = (0,x)=(s(x),x)$,
\item $\nu((a,x)) = (-a,x)=(i(a),x)$,
\end{itemize}
for all $a_1,a_2, a \in R_x$ and $x \in X$. Define $\zeta_{\mathcal{O}}: (Y, \rho_Y) \rightarrow (\mathscr{R} \rtimes X, \rho_{\mathscr{R} \rtimes X})$ by setting $\zeta_{\mathcal{O}}(a)=(a,x)$ for  $a \in R_x$. Clearly, $\zeta_{\mathcal{O}}$ is a bijection and the diagram
		\begin{center}
			\begin{tikzcd}
				(Y,\rho_Y) \arrow[r, "\zeta_{\mathcal{O}}"] \arrow[rd, "p"'] & (\mathscr{R}\rtimes X,\rho_{\mathscr{R}\rtimes X}) \arrow[d, "q"] \\
				& (X,\rho_X)
			\end{tikzcd}
		\end{center}
commutes. Note that $\zeta_{\mathcal{O}}(m(a,b))=\mu(\zeta_{\mathcal{O}}(a)$, $\zeta_{\mathcal{O}}(b)),~ \zeta_{\mathcal{O}}(s(a))=\sigma(\zeta_{\mathcal{O}}(a))$ and $\zeta_{\mathcal{O}}(i(a))=\nu(\zeta_{\mathcal{O}}(a))$ for all $a,b \in R_x$. Hence, $\zeta_{\mathcal{O}}$ respects the abelian group structure. Now, for $a \in R_x$	and $b \in R_y$, we have
	\begin{eqnarray*}
		\zeta_{\mathcal{O}}(a) \tilde{\ast} \zeta_{\mathcal{O}}(b) &=& (a,x) \tilde{\ast} (b,y) \\
		&=& \big(\phi_{x,y}(a)+\psi_{x,y}(b), x \ast y \big)\\
		&=& \big(a \star s(y)+ s(x) \star b, x \ast y \big)\\
		&=& \big(m \big(a \star s(y),s(x) \star b \big), x \ast y \big)\\
		&=& \big(m((a,s(x)) \tilde{\star} (s(y),b)),x \ast y\big)\\
		&=&\big(m(a,s(x)) \star m(s(y),b), x \ast y\big),\quad \textrm{since $m: Y \times Y \to Y$ is a rack homomorphism}\\
		&=& (a \star b,x \ast y)\\
		&=& \zeta_{\mathcal{O}}(a \star b),
	\end{eqnarray*}
where $\star$ is the rack operation on $Y$ and $\tilde{\star}$ is the rack operation on $Y \times Y$. Thus, $\zeta_{\mathcal{O}}$ is a rack homomorphism. Since $\rho_{\mathscr{R}\rtimes X} \, \zeta_{\mathcal{O}}=\zeta_{\mathcal{O}} \, \rho_Y$, it follows that $\zeta_{\mathcal{O}}$ is a symmetric rack isomorphism, and hence 
$\zeta_{\mathcal{O}}: \mathcal{O} \to \mathcal{T}\mathcal{S} (\mathcal{O})$ is an isomorphism.
\par 

Let $\mathcal{O}$ and $\mathcal{O'}$ be the objects $p:(Y,\rho_Y) \rightarrow (X, \rho_X)$ and $p':(Y',\rho_{Y'}) \rightarrow (X,\rho_X)$ in $ \Ab(\SR|_{(X,\rho_X)})$, respectively. Let $f:\mathcal{O} \rightarrow \mathcal{O'}$ be a morphism. Then, by definition, $f: (Y, \rho_Y) \to (Y', \rho_{Y'})$ is a symmetric rack homomorphism such that $p' \circ f=p$ and $f$ preserves the abelian group structure. Let $x \in X$ and $a \in R_x$. Then $\zeta_{\mathcal{O'}} \, f(a)=(f_x(a),x)$, where $f_x=f|_{R_x}$ and $$\mathcal{T} \mathcal{S}(f) \, \zeta_{\mathcal{O}} (a)=\mathcal{T} \mathcal{S}(f)(a,x)= \mathcal{T}(f|_{R_x}(a))=\mathcal{T}(f_x(a))=(f_x(a),x).$$ Thus, the diagram 
\begin{center}
	\begin{tikzcd}[row sep=3em, column sep=4em]
		\mathcal{O} \arrow[r, "\zeta_{\mathcal{O}}"] \arrow[d, "f"] & \mathcal{T}\mathcal{S}(\mathcal{O}) \arrow[d, "\mathcal{T}\mathcal{S}(f)"] \\
	\mathcal{O'} \arrow[r, "\zeta_{\mathcal{O'}}"] & \mathcal{T}\mathcal{S}(\mathcal{O'})
	\end{tikzcd}
\end{center}
commutes, and we have shown that $\mathcal{T}\mathcal{S}$ is naturally isomorphic to $\id_{ \Ab(\SR|_{(X,\rho_X)})}$.
\par 
Next, we prove that $\mathcal{S}\mathcal{T}$ is naturally isomorphic to $\id_{\mathbf{SRMod}_{(X, \rho_X)}}$. Let $\mathcal{O}$ be the object $\mathscr{F}=(A,\phi,\psi,\eta)$ in $\mathbf{SRMod}_{(X, \rho_X)}$. Then $\mathcal{T}(\mathcal{O})$ is the abelian group object $p:(\mathscr{F} \rtimes X,\rho_{\mathscr{F} \rtimes X}) \rightarrow (X, \rho_X)$ as in Proposition \ref{abelian-group-object}. For each $x \in X$, since $A'_{x}:=p^{-1}(x)=(A_x,x)$, it follows that $\mathcal{S}\mathcal{T}(\mathcal{O}) $ is the object $\mathscr{F'}:=(A',\phi',\psi',\eta')$, where
\begin{align}
	\phi_{x,y}' (a,x) &= (a,x) \tilde{\ast} \sigma(y)=\ (a,x) \tilde{\ast} (0,y)=(\phi_{x,y}(a),x \ast y),  \label{I}\\
	\psi_{x,y}' (b,y) &= \sigma(x) \tilde{\ast} (b,y)=(0,x) \tilde{\ast} (b,y) =(\psi_{x,y}(b),x \ast y), \label{II}\\ 
	\eta_x' (a,x)& = \rho_{\mathscr{F} \rtimes X}(a,x)=(\eta_x(a), \rho_X(x)), \label{III}
\end{align}	
for all $a\in A_x$, $b \in A_y$ and $x,y \in X$. 
\par 
We now construct an $(X,\rho_X)$-module isomorphism $\lambda_{\mathcal{O}}: \mathcal{O} \to \mathcal{S}\mathcal{T}(\mathcal{O})$. For each $ x \in X$, define $(\lambda_{\mathcal{O}})_x:A_x \to A'_x$ as $a \mapsto (a,x)$ for all $a \in A_x$. Clearly, $(\lambda_O)_{x}$ is an abelian group isomorphism for each $x \in X$. Also, using \eqref{I}, \eqref{II} and \eqref{III}, we have
\begin{eqnarray*}
	\phi'_{x,y}(\lambda_O)_{x}(a) &=& \phi'_{x,y}(a,x)= (\phi_{x,y}(a),x \ast y) =(\lambda_O)_{x\ast y} \phi_{x,y}(a),\\
	\psi'_{x,y}(\lambda_O)_{y}(b)&=&\psi'_{x,y}(b,y) = (\psi_{x,y}(b),x \ast y)= (\lambda_O)_{x\ast y}\psi_{x,y}(b),\\
	\eta'_{x}(\lambda_O)_{x}(a) &=& \eta'_x(a,x)=(\eta_x(a), \rho_X(x))=(\lambda_O)_{\rho_X(x)} \eta_{x}(a),
\end{eqnarray*}
for all $a\in A_x$, $b \in A_y$ and $x,y \in X$.  Thus, \hyperref[X-map]{(3.0.1-3.0.3)} hold, and $\lambda_O$ is an $(X, \rho_X)$-module isomorphism. Finally, let $\mathcal{O}_1$ and $\mathcal{O}_2$ be the objects $\mathscr{F}_1=(A,\phi,\psi,\eta)$ and  $\mathscr{F}_2=(\tilde{A},\tilde{\phi},\tilde{\psi},\tilde{\eta})$ in $\mathbf{SRMod}_{(X, \rho_X)}$, respectively.  Let $g:\mathcal{O}_1 \rightarrow \mathcal{O}_2$ be a $(X, \rho_X)$-map. It is easy to see that the diagram
\begin{center}
	\begin{tikzcd}[row sep=3em, column sep=4em]
\mathcal{O}_1 \arrow[r, "\lambda_{\mathcal{O}_1}"] \arrow[d, "g"] & \mathcal{S}\mathcal{T}(\mathcal{O}_1)  \arrow[d, "\mathcal{T}\mathcal{S}(g)"] \\
\mathcal{O}_2\arrow[r, "\lambda_{\mathcal{O}_2}"] & \mathcal{S}\mathcal{T}(\mathcal{O}_2)	\end{tikzcd}
\end{center}
 commutes.  Thus, $\mathcal{S}\mathcal{T}$ is naturally isomorphic to $\id_{\mathbf{SRMod}_{(X, \rho_X)}}$, and the categories $ \Ab(\SR|_{(X,\rho_X)})$ and $\mathbf{SRMod}_{(X, \rho_X)}$ are equivalent. 
	\end{proof}

We now turn our attention to symmetric quandles. We observe that, if $\mathscr{F}=(A,\phi,\psi,\eta)$ is a module over a symmetric quandle $(X,\rho_X)$, then by Proposition \ref{semi-direct product}, the semi-direct product $(\mathscr{F} \rtimes X,\rho_{\mathscr{F} \rtimes X})$ is a symmetric quandle. Let $p:(Y,\rho_Y) \rightarrow (X, \rho_X)$ be an abelian group object with the multiplication map $m$ and the section $s$,  as in Proposition \ref{module}. We see that
\begin{eqnarray*}
\phi_{x,x}(u)+\psi_{x,x}(u) &=& u \star s(x) + s(x)\star u\\
&=& m \big( u \star s(x), s(x)\star u\big)\\
&=& m \big( (u, s(x)) \star (s(x),  u) \big)\\
&=& m \big( (u, s(x) \big) \star m \big( (s(x),  u) \big)\\
&=& u \star u\\
&=& u
\end{eqnarray*}
for all $x \in X$ and $u \in R_x$, which is  condition \hyperref[M9]{(M9)}. Thus, we immediately obtain the following result.

\begin{theorem}\label{equivalence of categories for symmetric quandle modules}
Let $(X,\rho_X)$ be a symmetric quandle. Then the categories  $\Ab(\SQ|_{(X,\rho_X)})$ and $\mathbf{SQMod}_{(X, \rho_X)}$ are equivalent.
\end{theorem} 
\medskip


\section{Abelian extensions of symmetric racks and symmetric quandles}\label{section abelian extensions of symmetric racks}

Throughout this section, for a symmetric rack (respectively quandle) $(X,\rho_X)$, we denote its underlying rack (respectively quandle) by $(X, \ast)$. 

\begin{definition} Let $(X,\rho_X)$ be a symmetric rack (respectively quandle) and $\mathscr{F}=(A,\phi,\psi,\eta)$ be an $(X,\rho_X)$-module.
		An \textit{extension} of $(X,\rho_X)$ by $\mathscr{F}$ comprises of the following data:
		\begin{enumerate}
			\item[(E1)] A symmetric rack (respectively quandle) $(E,\rho_E)$, where $\star$ is the rack (respectively quandle) operation in $E$.
			\item[(E2)] An epimorphism $f:(E,\rho_E) \rightarrow (X,\rho_X)$. This epimorphism further induces a partition $E=\cup_{x\in X}E_x$, where $E_x:=f^{-1}(x)$ for $x \in X$.
			\item[(E3)] For each $x \in X$, a left $A_x$-action on $E_x$, written $(a, s) \mapsto a \cdot s$ for $a\in A_x$ and $s \in E_x$. Further, the action satisfies the following conditions for all $a\in A_x$, $b \in A_y$, $s \in E_x$ and $t \in E_y$:
			\begin{enumerate}
				\item The $A_x$-action on $E_x$ is free and transitive.
				\item $(a\cdot s)\star t=\phi_{x,y}(a)\cdot (s \star t)$.
				\item $s \star (b\cdot t)=\psi_{x,y}(b)\cdot (s \star t).$
				\item $\rho_E(a\cdot s)=\eta_x(a)\cdot \rho_E(s).$
			\end{enumerate}

		\end{enumerate}
	\end{definition}

For brevity, we denote an extension by $f:(E,\rho_E) \rightarrow (X, \rho_X)$. Observe that $\rho_E(E_x) \subset E_{\rho_X(x)}$ for each $x \in X$. Next, we introduce equivalence of such extensions.

\begin{definition}
	Let $(X,\rho_X)$ be a symmetric rack (respectively quandle) and $\mathscr{F}=(A,\phi,\psi,\eta)$ be an $(X,\rho_X)$-module.
Two extensions $f_1:(E_1,\rho_{E_1}) \rightarrow (X,\rho_X)$ and $f_2:(E_2,\rho_{E_2}) \rightarrow (X,\rho_X)$ of $(X,\rho_X)$ by $\mathscr{F}$ are {\it equivalent} if there exists a symmetric rack (respectively quandle) isomorphism, $\Phi:(E_1,\rho_{E_1}) \rightarrow (E_2,\rho_{E_2})$ which respects projection maps and group actions, that is,
	 \begin{enumerate}\label{equivalent extension}
	 	\item[(EE1)] $f_2 \Phi(s)=f_1(s) $ for all $s \in E_1$.
	 	\item[(EE2)] $\Phi(a\cdot s)=a \cdot \Phi(s)$ for all $s \in E_x$, $a \in A_x$ and $x \in X$.
	 \end{enumerate}
\end{definition}
	
Let $f:(E,\rho_E) \rightarrow (X, \rho_X)$ be an extension of $(X,\rho_X)$ by $\mathscr{F}$. Let $(E, \star)$ and $(X, *)$ be the underlying racks. Then, a {\it section} $s:(X,\rho_X) \rightarrow (E,\rho_E)$ is a map such that $fs=\id_X$ and $s \rho_X=\rho_E s$. 	Since, $f$ is a symmetric rack homomorphism, it follows that $s(x) \star s(y)\in E_{x \ast y}$. Further, since $A_{x\ast y}$ acts freely and transitively on $E_{x \ast y}$, there exists a unique $\sigma_{x,y} \in A_{x \ast y}$ such that
$$s(x) \star s(y)= \sigma_{x,y}\cdot s(x \ast y).$$

This leads to the following definition.
	
	\begin{definition}
		Let $(X,\rho_X)$ be a symmetric rack (respectively quandle) and $\mathscr{F}=(A,\phi,\psi,\eta)$ be an $(X,\rho_X)$-module.
		Given an extension $f:(E,\rho_E) \rightarrow (X, \rho_X)$ of $(X, \rho)$ by $\mathscr{F}$ and a section $s:(X,\rho_X) \rightarrow (E,\rho_E)$, a \textit{factor set} is a map $\sigma:X \times X \rightarrow \bigsqcup_{x \in X} A_x$ such that
\begin{equation}\label{factor set}
s(x) \star s(y)= \sigma_{x,y} \cdot s(x \ast y)
\end{equation}
 for all $x,y \in X$, where we write $\sigma(x,y)=\sigma_{x,y}$.
	\end{definition}
	
Note that when $\sigma_{x,y}=0 \in A_{x*y}$, then $s:(X,\rho_X) \rightarrow (E,\rho_E)$ is a symmetric rack homomorphism. Thus, $\sigma$ can be regarded as an obstruction to $s$ being a symmetric rack homomorphism. For each $x \in X$, since the $A_x$-action on $E_x$ is free and transitive, every element of $E_x$ can be written uniquely in the form 
$a\cdot s(x)$ for some unique $a \in A_x$. Thus, for $a \in A_x$ and $b\in A_y$, we have
\begin{eqnarray*}
\big(a\cdot s(x)\big) \star \big(b\cdot s(y)\big)&=& \phi_{x,y}(a)\cdot \big(s(x) \star (b\cdot s(y))\big)\\
&=&\phi_{x,y}(a)\cdot (\psi_{x,y}(b)\cdot \big(s(x)\star s(y))\big)\\
&=&\phi_{x,y}(a)\cdot (\psi_{x,y}(b)\cdot \big(\sigma_{x,y} \cdot s(x \ast y))\big)\\
&=& \big(\phi_{x,y}(a)+\psi_{x,y}(b)+\sigma_{x,y})\cdot (s(x\ast y)\big).
\end{eqnarray*}
	Thus, the rack structure on $E$ is completely determined by the factor set $\sigma$.
	
	\begin{theorem}\label{extension}
		Let $(X,\rho_X)$ be a symmetric rack and $\mathscr{F}=(A,\phi,\psi,\eta)$ be an $(X,\rho_X)$-module. Let $\sigma:X \times X \rightarrow \bigsqcup_{x \in X} A_x$ be a map and ${E}(\mathscr{F},\sigma) := \big\{(a,x) \mid x\in X, ~a \in A_x \big\}$. Define the binary operation
		$$(a,x) \tilde{\star} (b,y) := \big(\phi_{x,y}(a)+\psi_{x,y}(b)+\sigma_{x,y},x\ast y \big)$$
on ${E}(\mathscr{F},\sigma)$ and the map  $\rho_{E(\mathscr{F},\sigma)}:E(\mathscr{F},\sigma) \to E(\mathscr{F},\sigma)$ given by
$$ (a,x) \mapsto \big(\eta_x(a),\rho_X(x)\big)$$
for all $a \in A_x, b \in A_y$ and $x,y \in X$. Then, $(E(\mathscr{F},\sigma),\rho_{E(\mathscr{F},\sigma)})$ is an extension of $(X,\rho_X)$ by $\mathscr{F}$ with factor set $\sigma$ if and only if 
		\begin{enumerate}
		\item [$\mathrm{(F1)}$]$\sigma_{x \ast y,z}+\phi_{x \ast y,z}(\sigma_{x,y})=\phi_{x \ast z, y \ast z}(\sigma_{x,z})+\sigma_{x \ast z, y \ast z}+\psi_{x \ast z, y \ast z}(\sigma_{y,z})$,
		\item[$\mathrm{(F2)}$] $\sigma_{\rho_X(x),y}=\eta_{x\ast y}(\sigma_{x,y})$,
		\item[$\mathrm{(F3)}$] $\phi_{x\ast \rho_X(y),y}(\sigma_{x,\rho_X(y)})=-\sigma_{x \ast \rho_X(y),y}$,	
	\end{enumerate}	
for all $x,y,z \in X$. Moreover, if $(E, \rho_E)$ is an extension of $(X, \rho_X)$ by $\mathscr{F}$ with factor set $\sigma$, then $\mathrm{(F1)-(F3)}$ hold and $(E,\rho_E)$ is equivalent to $E(\mathscr{F},\sigma).$
\par

Further, if $(X, \rho_X)$ is a symmetric quandle, then we have the additional condition
\begin{enumerate}
\item[$\mathrm{(F4)}$] $\sigma_{x, x}=0$ for all $x \in X$.
\end{enumerate}	
\end{theorem}

\begin{proof}
Suppose that $\mathrm{(F1)-(F3)}$ hold. That $(E(\mathscr{F},\sigma),\rho_{E(\mathscr{F},\sigma)})$ is a rack is proved in \cite[Proposition 3.1]{MR2155522}. We check that $\rho_{E(\mathscr{F},\sigma)}$ is a good involution.
\begin{itemize}
\item For (S1), we have $$\rho_{E(\mathscr{F},\sigma)}^2((a,x))=\rho_{E(\mathscr{F},\sigma)}(\eta_x(a), \rho_X(x))=(\eta_{\rho_X(x)}(\eta_x(a)),\rho_X^2(x))=(a,x).$$
\item For (S2), we see that 
\begin{eqnarray*}
\big(\rho_{E(\mathscr{F},\sigma)}(a,x) \big)\tilde{\star} (b,y) &=& \big(\eta_x(a), \, \rho_X(x)\big) \tilde{\star} (b,y)\\
&=& \big(\phi_{\rho_X(x),y}(\eta_x(a))+\psi_{\rho_X(x),y}(b)+\sigma_{\rho_X(x),y}, \,\rho_X(x) \ast y \big)\\
&=&\big( \eta_{x \ast y}(\phi_{x,y}(a)+\psi_{x,y}(b)+\sigma_{x,y}), \, \rho_X(x \ast y) \big), \quad \textrm{by (M4), (M5) and (F2)}\\
&=& \rho_{E(\mathscr{F},\sigma)} \big(\phi_{x,y}(a)+\psi_{x,y}(b)+\sigma_{x,y}, \, x \ast y \big)\\
&=& \rho_{E(\mathscr{F},\sigma)} \big((a,x)\tilde{\star} (b,y) \big).
\end{eqnarray*}
\item For (S3), we have 
\begin{eqnarray*}
(a,x) \tilde{\star} \rho_{E(\mathscr{F},\sigma)}(b,y) &=& (a,x)\tilde{\star} \big(\eta_y(b), \,\rho_X(y) \big)\\
&=& \big(\phi_{x,\rho_X(y)}(a)+\psi_{x, \rho_X(y)}(\eta_y(b))+\sigma_{x, \rho_{X}(y)}, \, x \ast \rho_X(y) \big)\\
&=& \big(\phi_{x \ast^{-1}y,y}^{-1}(a-\sigma_{x\ast^{-1}y,y}-\psi_{x \ast^{-1}y,y}(b)), \, x \ast^{-1}y \big), \quad \textrm{by (M6), (M8) and (F3)}\\
 &=& (a,x)\tilde{\star}^{-1}(b,y).
 \end{eqnarray*}
\end{itemize}
Hence,  $(E(\mathscr{F},\sigma),\rho_{E(\mathscr{F},\sigma)})$ is a symmetric rack.	Define $f_1:(E(\mathscr{F},\sigma),\rho_{E(\mathscr{F},\sigma)}) \rightarrow (X,\rho_X)$ to be the projection onto the second coordinate. Let $A_x$ act on $E(\mathscr{F},\sigma)_x =f_1^{-1}(x)$ by $a\cdot (b,x)=(a+b,x)$ for each $x \in X$ and $a, b \in A_x$. Note that these actions are free and transitive. Also, as proved in \cite[Proposition 3.1]{MR2155522}, the actions satisfy
$$\big(\alpha\cdot (a,x) \big)\tilde{\star} (b,y)=\phi_{x,y}(\alpha)\cdot \big((a,x)\tilde{\star}(b,y) \big) \quad \textrm{and} \quad (a,x) \tilde{\star} \, \big(\beta \cdot (b,y) \big)= \psi_{x,y}(\beta)\cdot \big((a,x)\tilde{\star}(b,y)\big)$$
for all $\alpha \in A_x$, $\beta \in A_y$, $(a,x) \in E(\mathscr{F},\sigma)_x$ and $(b, y) \in E(\mathscr{F},\sigma)_x$. Further,  we have $$\rho_{E(\mathscr{F},\sigma)} \big(\alpha\cdot (a,x) \big)=\rho_{E(\mathscr{F},\sigma)} \big((\alpha+a,x)\big)=\big(\eta_x(\alpha+a),\rho_X(x)\big)=\eta_x(\alpha)\cdot (\rho_{E(\mathscr{F},\sigma)}(a,x)).$$ Thus, $(E(\mathscr{F},\sigma),\rho_{E(\mathscr{F},\sigma)})$ is an extension of $(X,\rho_X)$ by $\mathscr{F}$.	Define $s:X \rightarrow E(\mathscr{F},\sigma)$ as $s(x )= (0,x)$. This is clearly a section satisfying
	$$s(x) \tilde{\star} s(y)=(0,x) \ast (0,y)= (\sigma_{x,y},x*y)= \sigma_{x,y} \cdot s(x \ast y)$$ and
	$$\rho_{E(\mathscr{F},\sigma)}(s(x))=\rho_{E(\mathscr{F},\sigma)}(0,x)=(0,\rho_X(x))=s(\rho_X(x)).$$
	Thus, $\sigma$ is the factor set corresponding to the section s. The converse follows by reversing the steps. 
\par

For the second assertion, let $f_2:(E,\rho_E) \rightarrow (X, \rho_X)$ be an extension of $(X,\rho_X)$ by $\mathscr{F}$ with factor set $\sigma$ corresponding to a section $s:(X,\rho) \rightarrow (E,\rho_E)$. Define $$\Phi:(E(\mathscr{F},\sigma),\rho_{E(\mathscr{F},\sigma)})\rightarrow (E,\rho_E)$$ as $\Phi \big((a,x)\big)= a\cdot s(x)$.
	Since, $A_x$ acts freely and transitively, it follows that $\Phi$ is a bijection. Also, using (E3), we have
\begin{eqnarray*}
		\Phi \big((a,x)\tilde{\star}(b,y) \big)&=&\Phi \big(\phi_{x,y}(a)+\psi_{x,y}(b)+\sigma_{x,y}, \, x\ast y\big)\\
		&=&(\phi_{x,y}(a)+\psi_{x,y}(b)+\sigma_{x,y})\cdot s(x \ast y)\\
		&=& (\phi_{x,y}(a)+\psi_{x,y}(b)) \cdot (s(x)\tilde{\star} s(y))\\
		&=& \phi_{x,y}(a) \cdot \big(s(x)\tilde{\star} (b\cdot s(y)) \big)\\
		&=&(a\cdot s(x))\tilde{\star} (b \cdot s(y))\\
		&=& \Phi(a,x) \tilde{\star} \Phi(b,y).
	\end{eqnarray*}
Hence, $\Phi$ is a rack isomorphism. Further, we see that 
\begin{eqnarray*}
\Phi \rho_{E(\mathscr{F},\sigma)}(a,x) &=& \Phi \big(\eta_x(a),\rho_X(x) \big)\\
&=& \eta_x(a)\cdot s(\rho_X(x))\\
&=& \eta_x(a)\cdot \rho_E(s(x)), \quad \textrm{since $\rho_E \, s=s \, \rho_X$}\\
&=& \rho_E(a\cdot s(x)), \quad \textrm{by (E3)}\\
&=& \rho_E \Phi(a,x)
\end{eqnarray*}
for all $x \in X$ and $a \in A_x$. Hence, $\Phi$ is an isomorphism of symmetric racks. Clearly, $f_2 \Phi(a, x)=f_2(a \cdot s(x))=x=f_1(a, x)$ for all $(a, x) \in E(\mathscr{F},\sigma)_x$ and $\Phi(\alpha \cdot (a,x))=\Phi((\alpha+a,x))=(\alpha+a)\cdot s(x)=\alpha\cdot \Phi(a,x)$ for all $\alpha, a \in A_x$ and $x \in X$. Hence, $\Phi$ is an equivalence of extensions.
\par
Finally, if $(X, \rho_X)$ is a symmetric quandle, then the condition (F4) follows.
\end{proof}

\begin{proposition}\label{equivalence}
	Let $(X,\rho_X)$ be a symmetric rack (respectively quandle) and $\mathscr{F}$ an $(X,\rho_X)$-module. Let $\sigma$ and $\tau$ be factor sets corresponding to sections $s$ and $t$, respectively, for extensions of $(X,\rho_X)$ by $\mathscr{F}$. Then the following assertions are equivalent:
	\begin{enumerate}
		\item The extensions $(E(\mathscr{F},\sigma),\rho_{E(\mathscr{F},\sigma)})$ and $(E(\mathscr{F},\tau),\rho_{E(\mathscr{F},\tau)})$ are equivalent.
		\item There exists a set of elements $\{v_x\in A_x \mid x \in X\}$ such that $\tau_{x,y}=\sigma_{x,y}+\phi_{x,y}(v_x)+\psi_{x,y}(v_y)-v_{x\ast y}$ and $\eta_x(v_x)=v_{\rho_X(x)}$ for $x,y \in X$.
		\item $\sigma$ and $\tau$ are factor sets of the same extension of $(X, \rho_X)$ by $\mathscr{F}$ corresponding to different sections $s$ and $t$, respectively.
	\end{enumerate} 
\end{proposition}
	
\begin{proof}
First, we prove $(1) \Leftrightarrow (2)$. Let $\Phi:(E(\mathscr{F},\sigma),\rho_{E(\mathscr{F},\sigma)}) \rightarrow (E(\mathscr{F},\tau),\rho_{E(\mathscr{F},\tau)})$ be an equivalence of extensions. Since $\Phi$ respects the projections, we obtain $\Phi(0,x) = (v_x,x)$ for some $v_x \in A_x$.	Further, since $\Phi$ respects the group actions, this gives $\Phi(a,x)=\Phi(a \cdot (0,x))=a \cdot \Phi(0,x)=a \cdot (v_x,x)=(a+v_x,x)$ for all $a \in A_x$. Then,
	$$\Phi \big((a,x)\tilde{\star}(b,y) \big)=\big(\phi_{x,y}(a)+\psi_{x,y}(b)+\tau_{x,y}+v_{x \ast y}, \, x \ast y \big)$$ and
    $$\Phi(a,x)\bar{\star} \Phi(b,y)=(a+v_x,x)\bar{\star} (b+v_y,y)=\big(\phi_{x,y}(a+v_x)+\psi_{x,y}(b+v_y)+\sigma_{x,y}, \, x \ast y \big),$$
    where $\tilde{\star}$ and $\bar{\star}$ are the rack operations in $E(\mathscr{F},\sigma)$  and $E(\mathscr{F},\tau)$, respectively.     Thus, $\Phi \big((a,x)\bar{\star}(b,y) \big)=\Phi(a,x)\bar{\star} \Phi(b,y)$ if and only if $\tau_{x,y}=\sigma_{x,y}+\phi_{x,y}(v_x)+\psi_{x,y}(v_y)-v_{x\ast y}$ for $x,y \in X$. We have
    $$\Phi \rho_{E(\mathscr{F},\sigma)}(a,x)=\Phi(\eta_x(a),\rho_X(x))=(\eta_x(a)+v_{\rho_X(x)},\rho_X(x))$$
    and
    $$\rho_{E(\mathscr{F},\tau)}\Phi(a,x)=\rho_{E(\mathscr{F},\tau)}(a+v_x,x)=(\eta_x(a+v_x),\rho_X(x)).$$
    Thus, $\Phi \rho_{E(\mathscr{F},\sigma)}=\rho_{E(\mathscr{F},\tau)}\Phi$ if and only if $\eta_x(v_x)=v_{\rho_X(x)}$.
\par    
    Next, we prove $(1) \Leftrightarrow (3)$. We shall use $(1) \Leftrightarrow (2)$ proved above. Obviously, $\sigma$ is the factor set corresponding to the section $s:(X,\rho_X) \rightarrow (E(\mathscr{F},\sigma),\rho_{E(\mathscr{F},\sigma)})$ given by $s(x) = (0,x)$. Define  $t:(X,\rho_X) \rightarrow (E(\mathscr{F},\sigma),\rho_{E(\mathscr{F},\sigma)})$ by $t(x) = (v_x,x)$. Since $$t(\rho_X(x))=(v_{\rho_X(x)},\rho_X(x))=(\eta_x(v_x),\rho_X(x))=\rho_{E(\mathscr{F},\sigma)}(t(x)),$$ it follows that $t$ is a section. Note that, we have $\tau_{x,y}=\sigma_{x,y}+\phi_{x,y}(v_x)+\psi_{x,y}(v_y)-v_{x\ast y}$. Then, we see that
\begin{eqnarray*}
t(x) \tilde{\star} t(y) &=& (v_x,x) \tilde{\star} (v_y,y)\\
&=& (\phi_{x, y}(v_x) + \psi_{x, y} (v_y) + \sigma_{x, y}, x*y)\\
&=&  (\tau_{x, y} + v_{x*y}, x*y)\\
&=&\tau_{x, y} \cdot (v_{x*y}, x*y)\\
&=& \tau_{x, y}  \cdot t(x*y).
\end{eqnarray*}    
Thus, $\tau$ is the factor set corresponding to $t$.
    \par
    Conversely, suppose that $\sigma$ and $\tau$ are factor sets of some extension $(E,\rho_E)$ of $(X,\rho_X)$ by $\mathscr{F}$ corresponding to sections $s: (X,\rho_X) \rightarrow (E,\rho_E)$ and $t:(X,\rho_X) \rightarrow (E,\rho_E)$, respectively. Then, for each $x \in X$, we have $s(x)=v_x \cdot t(x)$ for a unique $v_x \in A_x$.   Define $\Phi:E(\mathscr{F},\sigma) \rightarrow E(\mathscr{F},\tau)$ as $\Phi \big((a,x) \big)= (a+v_x,x)$. Then, a direct check implies that $\Phi$ is an equivalence of extensions.
\end{proof}	
	
\begin{corollary}\label{split}
	Let $(X,\rho_X)$ be a symmetric rack (respectively quandle) and $\mathscr{F}=(A,\phi,\psi,\eta)$ be an $(X,\rho_X)$-module.
	Given an extension $f:(E,\rho_E) \rightarrow (X, \rho_X)$ of $(X,\rho_X)$ by $\mathscr{F}$, the following statements are equivalent:
	\begin{enumerate}
		\item There exists a symmetric-rack homomorphism $s:(X, \rho_X) \rightarrow (E,\rho_E)$ such that $f \, s=\id_X$ and $s \,\rho_X=\rho_E \,s.$
		\item There exists a section such that its corresponding factor set is trivial.
		\item For any section and the corresponding factor set $\sigma$, there exists a set of elements $\{v_x\in A_x \mid x\in X\}$ such that
		$$\sigma_{x,y}=\phi_{x,y}(v_x)+\psi_{x,y}(v_y)-v_{x\ast y}\quad  \textrm{and} \quad \eta_x(v_x)=v_{\rho_X(x)}$$
	\end{enumerate}
 for all $x,y \in X$.
\end{corollary}
Extensions satisfying any of (1), (2) or (3) in Corollary \ref{split} are called  \textit{split extensions}.

\begin{definition}
	Given a symmetric rack (respectively quandle) $(X,\rho_X)$ and an $(X,\rho_X)$-module $\mathscr{F}$, a map $\sigma:X \times X \rightarrow \bigsqcup_{x \in X} A_x$ is called a \textit{coboundary} if there exists a map $v:X \rightarrow \bigsqcup_{x \in X} A_x$ such that
$$ \sigma_{x,y}=\phi_{x,y}(v_x)+\psi_{x,y}(v_y)-v_{x\ast y} \quad \textrm{and} \quad	\eta_x(v_x)=v_{\rho_X(x)}$$
for all $x,y \in X$.
\end{definition}

\begin{theorem}\label{theorem ext cohom}
	Let $(X, \rho_X)$ be a symmetric rack and $\mathscr{F}=(A,\phi,\psi,\eta)$ be an $(X,\rho_X)$-module. Then there is an abelian group $\mathcal{H}^2_{\mathrm{SR}}(X,\mathscr{F})$ whose elements are in bijective correspondence with the set of equivalence classes of extensions of $(X, \rho_X)$ by $\mathscr{F}$. 
\end{theorem}

\begin{proof}
Let $Z^2(X,\mathscr{F})$ consists of all factor sets $\sigma:X \times X \rightarrow \bigsqcup_{x \in X} A_x$. Define the binary operation $$(\sigma+\tau)_{x,y}:=\sigma_{x,y}+\tau_{x,y}$$
for all $x, y \in X$. In view of Theorem \ref{extension}, it follows that the conditions (F1)-(F3) are satisfied by $\sigma+\tau$, and hence this binary operation turns
$Z^2(X,\mathscr{F})$  into an abelian group with the trivial factor set as the identity element.  Further, the set $B^2(X,\mathscr{F})$ of all coboundaries forms a subgroup of $Z^2(X,\mathscr{F})$. Define $\mathcal{H}^2_{\mathrm{SR}}(X,\mathscr{F}):=Z^2(X,\mathscr{F})/B^2(X,\mathscr{F})$. Then, by Theorem \ref{extension} and Proposition \ref{equivalence}, $\mathcal{H}^2_{\mathrm{SR}}(X,\mathscr{F})$ is in bijective correspondence with the set of equivalence classes of extensions of $(X,\rho_X)$ by $\mathscr{F}$.  
\end{proof}

Using Theorem \ref{extension} and \cite[Proposition 2.5]{MR2155522}, we obtain an analogous result for symmetric quandles.

\begin{theorem}\label{theorem ext cohom quandle}
	Let $(X, \rho_X)$ be a symmetric quandle and $\mathscr{F}=(A,\phi,\psi,\eta)$ be an $(X,\rho_X)$-module. Then there is an abelian group $\mathcal{H}^2_{\mathrm{SQ}}(X,\mathscr{F})$ whose elements are in bijective correspondence with the set of equivalence classes of extensions of $(X, \rho_X)$ by $\mathscr{F}$. 
	\end{theorem}
\medskip


\section{Generalized (co)homology for symmetric racks and symmetric quandles}\label{section generalized cohomology}

In this section, we formulate a (co)homology theory for symmetric racks (respectively quandles) $(X,\rho_X)$ with coefficients derived from homogeneous $(X,\rho_X)$-modules. To initiate this formulation, we first extract the fundamental axioms that define $(X,\rho_X)$-modules. 

\begin{definition}
		The \textit{symmetric rack algebra} of a symmetric rack $(X, \rho_X)$ is the associative $\mathbb{Z}$-algebra, denoted by $\mathbb{Z}(X,\rho_X)$, which is generated by the set $\{\phi_{x,y}^{\pm 1}, \, \psi_{x,y}, \,\eta_x \mid x,y \in X \}$ and admit the following defining relations: 
		\begin{enumerate}\label{axiom}
			\item[(A1)] $\phi_{x,y} \phi_{x,y}^{-1}=\phi_{x,y}^{-1}\phi_{x,y}=1$,
			\item[(A2)] $\phi_{x \ast y,z} \phi_{x,y}=\phi_{x \ast z, y \ast z}\phi_{x,z}$,
			\item[(A3)] $\phi_{x \ast y,z}\psi_{x,y}=\psi_{x \ast z, y \ast z}\phi_{y,z}$,
			\item[(A4)] $\eta_{\rho_X(x)} \eta_x=1$,
			\item[(A5)] $\phi_{\rho_X(x),y} \eta_x= \eta_{x\ast y} \phi_{x,y}$,
			\item[(A6)] $\psi_{\rho_X(x),y}= \eta_{x \ast y} \psi_{x,y}$,
			\item[(A7)] $\phi_{x \ast \rho_X(y),y} \phi_{x, \rho_X(y)}=1$,
			\item[(A8)] $\psi_{x \ast y,z}=\phi_{x \ast z, y \ast z} \psi_{x,z}+\psi_{x \ast z, y \ast z} \psi_{y,z}$,
			\item[(A9)] $\phi_{x \ast \rho_X(y),y} \psi_{x, \rho_X(y)} \eta_y =- \psi_{x \ast \rho_X(y), y}$,
		\end{enumerate}
	for all $x,y,z \in X.$
	\end{definition}

\begin{definition}
The \textit{symmetric quandle algebra} of a symmetric quandle $(X, \rho_X)$ is the symmetric rack algebra of the underlying symmetric rack which satisfies the additional relation
		\begin{enumerate}\label{Qaxiom}
			\item [(A10)] $\phi_{x,x}+\psi_{x,x}=1$
		\end{enumerate}
	for all $x \in X.$
\end{definition}

\begin{remark}\label{Module-algebra}
	Let $(X,\rho_X)$ be a symmetric rack (respectively quandle). Then the category of homogeneous $(X,\rho_X)$-modules is equivalent to the category of modules over $\mathbb{Z}(X,\rho_X)$. Thus,  these are abelian with enough injective and projective objects.
\end{remark}

Let $(X,\rho_X)$ be a symmetric rack. For each $n\geq0$, define $C_n(X,\rho_X):=\mathbb{Z}(X,\rho_X)X^n,~\ie$, the free left $\mathbb{Z}(X,\rho_X)$-module with basis $X^n$, where $X^0=\{p\}$ for some fixed $p \in X$. For the sake of convenience, we fix the notation 
$$[x_1 \cdots x_n]:=(((x_1 \ast x_2) \ast x_3) \ast \cdots \ast x_n)$$ for any $(x_1, x_2, \ldots, x_n) \in X^n$. For each $n \geq 2$, let
$\partial_n:C_{n}(X,\rho_X) \rightarrow C_{n-1}(X, \rho_X)$ be the $\mathbb{Z}(X,\rho_X)$-linear map defined on the basis as
 \begin{eqnarray*}
\partial_n \big((x_1, \ldots, x_n) \big) &= &\sum_{j=2}^{n}(-1)^j (-1)^n \phi_{[x_1 \cdots \widehat{x_j} \cdots x_n],[x_j \cdots  x_n]}
  	(x_1,\ldots, \widehat{x_j}, \ldots , x_n)\\
  	& - &\sum_{j=2}^{n}(-1)^j (-1)^n (x_1 \ast x_j,\ldots,x_{j-1} \ast x_j, x_{j+1}, \ldots , x_n)\\
  	& + &(-1)^n \psi_{[x_1 \widehat{x_2}x_3 \cdots x_n],[x_2 x_3 \cdots x_n]}(x_2,\ldots , x_n),
\end{eqnarray*}
and 
$$\partial_1(x)=-\psi_{x \ast^{-1} p,p}(p).$$

Using conditions (A2), (A3), and (A8), it has been proved in \cite[Lemma 4.2]{MR1994219} that $\partial_{n-1} \, \partial_n=1$ for all $n \ge 1$.  Thus, $(C_n(X,\rho_X),\partial)$ is a chain complex. We next define a subcomplex of this chain complex.
\par   

For each $n \ge 1$, consider the left $\mathbb{Z}(X,\rho_X)$-submodule $D_n^{SR}(X,\rho_X)$ of $C_n(X,\rho_X)$ generated by $U_n \cup V_n$, where
$$U_n= \big\{ \eta_{[x_1 \cdots x_n]}(x_1, \ldots, x_n)- (\rho_X(x_1),x_2, \ldots ,x_n ) \mid (x_1, x_2, \ldots, x_n) \in X^n \big\},$$
\begin{eqnarray*}
V_n &=& \bigcup_{i=2}^n \big\{\big(\phi_{[x_1 \cdots \widehat{x_i} \cdots x_n],[x_i \cdots x_n]} (x_1 \ast x_i, \ldots , x_{i-1} \ast x_i, \rho_X(x_i),x_{i+1}, \ldots, x_n)+(x_1, \ldots, x_n) \big) \mid\\
&&  (x_1, x_2, \ldots, x_n) \in X^n \big\}.
\end{eqnarray*}
for $n \ge 2$ and $V_1=\emptyset$.

\begin{lemma}\label{partI}
Let $(X,\rho_X)$ be a symmetric rack. Then $\partial_n \big(U_n \big) \subset D_{n-1}^{SR}(X,\rho_X)$ for each $n \ge 2$.
  	\end{lemma}
\begin{proof}
Using the fact that $\rho_X$ is a good involution and the rack identity 
\begin{equation}\label{consequence rack identity}
[x_1 \cdots \widehat{x_j} \cdots x_n] \ast [x_j \cdots  x_n]= [x_1 \cdots x_n]
\end{equation}
for each $(x_1, x_2, \ldots, x_n) \in X^n$, it follows from the condition \hyperref[axiom]{(A5)} that
\begin{equation}\label{lemma eq 1}
\eta_{\underbrace{[x_1 \cdots x_n]}_{x \ast y}} \phi_{\underbrace{[x_1 \cdots \widehat{x_j} \cdots x_n]}_{x},\underbrace{[x_j \cdots  x_n]}_{y}} = \phi_{\underbrace{[\rho_X(x_1) \cdots \widehat{x_j} \cdots x_n]}_{\rho_X(x)},\underbrace{[x_j \cdots  x_n]}_{y}} \eta_{\underbrace{[x_1 \cdots \widehat{x_j} \cdots x_n]}_{x}}.
\end{equation}
Similarly, it follows from \hyperref[axiom]{(A6)} that
\begin{equation}\label{lemma eq 2}
\eta_{\underbrace{[x_1 \cdots x_n]}_{x \ast y}} \psi_{\underbrace{[x_1 \widehat{x_2}x_3 \cdots x_n]}_{x},\underbrace{[x_2 x_3 \cdots x_n]}_{y}}=\psi_{\underbrace{[\rho_X(x_1) \widehat{x_2}x_3 \cdots x_n]}_{\rho_X(x)},\underbrace{[x_2 x_3 \cdots x_n]}_{y}}.
\end{equation}
Then, we have
\begin{small}
\begin{eqnarray*}
&& \partial_n \big(\eta_{[x_1 \cdots x_n]}(x_1, \ldots, x_n)-(\rho_X(x_1),x_2, \ldots ,x_n) \big)\\
  		&=&\eta_{[x_1 \cdots x_n]}\partial_n(x_1, \ldots, x_n)-\partial_n (\rho_X\big(x_1),x_2, \ldots ,x_n), \quad \textrm{since $\partial_n$ is $\mathbb{Z}(X,\rho_X)$-linear}\\
  		& = &\sum_{j=2}^{n}(-1)^j (-1)^n \eta_{[x_1 \cdots x_n]} \phi_{[x_1 \cdots \widehat{x_j} \cdots x_n],[x_j \cdots  x_n]}
  		(x_1,\ldots, \widehat{x_j}, \ldots , x_n)\\
  		& - &\sum_{j=2}^{n}(-1)^j (-1)^n \eta_{[x_1 \cdots x_n]}
  		(x_1 \ast x_j,\ldots,x_{j-1} \ast x_j, x_{j+1}, \ldots , x_n)\\
  		& + &(-1)^n \eta_{[x_1 \cdots x_n]} \psi_{[x_1 \widehat{x_2}x_3 \cdots x_n],[x_2 x_3 \cdots x_n]}(x_2,\ldots , x_n)\\
  		& - &\sum_{j=2}^{n}(-1)^j (-1)^n \phi_{[\rho_X(x_1) \cdots \widehat{x_j} \cdots x_n],[x_j \cdots  x_n]}
  		(\rho_X(x_1),\ldots, \widehat{x_j}, \ldots , x_n)\\
  		& + &\sum_{j=2}^{n}(-1)^j (-1)^n
  		(\rho_X(x_1) \ast x_j,\ldots,x_{j-1} \ast x_j, x_{j+1}, \ldots , x_n)\\
  		& - &(-1)^n \psi_{[\rho_X(x_1) \widehat{x_2}x_3 \cdots x_n],[x_2 x_3 \cdots x_n]}(x_2,\ldots , x_n)\\
  		& = &\sum_{j=2}^{n}(-1)^j (-1)^n \Big(\eta_{[x_1 \cdots x_n]} \phi_{[x_1 \cdots \widehat{x_j} \cdots x_n],[x_j \cdots  x_n]}
  		(x_1,\ldots, \widehat{x_j}, \ldots , x_n)\\
  		& - & \phi_{[\rho_X(x_1) \cdots \widehat{x_j} \cdots x_n],[x_j \cdots  x_n]}
  		(\rho_X(x_1),\ldots, \widehat{x_j}, \ldots , x_n) \Big)\\
  		& - &\sum_{j=2}^{n}(-1)^j (-1)^n \Big( \eta_{[x_1 \cdots x_n]}
  		(x_1 \ast x_j,\ldots,x_{j-1} \ast x_j, x_{j+1}, \ldots , x_n)\\
  		& + &(\rho_X(x_1) \ast x_j,\ldots,x_{j-1} \ast x_j, x_{j+1}, \ldots , x_n) \Big)\\
  		& + &(-1)^n \Big(\eta_{[x_1 \cdots x_n]} \psi_{[x_1 \widehat{x_2}x_3 \cdots x_n],[x_2 x_3 \cdots x_n]}(x_2,\ldots , x_n)
  		- \psi_{[\rho_X(x_1) \widehat{x_2}x_3 \cdots x_n],[x_2 x_3 \cdots x_n]}(x_2,\ldots , x_n) \Big)\\
  		&=&\sum_{j=2}^{n}(-1)^j (-1)^n \phi_{[\rho_X(x_1) \cdots \widehat{x_j} \cdots x_n],[x_j \cdots  x_n]} \Big(\eta_{[x_1 \cdots \widehat{x_j} \cdots x_n]}
  		(x_1,\ldots, \widehat{x_j}, \ldots , x_n) -  (\rho_X(x_1),\ldots, \widehat{x_j}, \ldots , x_n)\Big)\\
  		& - &\sum_{j=2}^{n}(-1)^j (-1)^n \Big( \eta_{[x_1 \cdots x_n]} (x_1 \ast x_j,\ldots,x_{j-1} \ast x_j, x_{j+1}, \ldots , x_n) +(\rho_X(x_1) \ast x_j,\ldots,x_{j-1} \ast x_j, x_{j+1}, \ldots , x_n)\Big),\\
		&& \quad \textrm{using \eqref{lemma eq 1} and \eqref{lemma eq 2}}. 
\end{eqnarray*}
\end{small}
Hence, $\partial_n \big(U_n \big)$ lies in $D_{n-1}^{SR}(X,\rho_X)$. 
 \end{proof}	

  \begin{lemma}\label{partII}
  	Let $(X,\rho_X)$ be a symmetric rack. Then $\partial_n \big(V_n \big) \subset D_{n-1}^{SR}(X,\rho_X)$ for each $n \ge 2$.
  \end{lemma}

 \begin{proof}
  	For a fixed $2 \le i \le n$, we have
\begin{small}
\begin{eqnarray*}
&&   	\partial_n \big(\phi_{[x_1 \cdots \widehat{x_i} \cdots x_n],[x_i \cdots x_n]}(x_1 \ast x_i, \ldots , x_{i-1} \ast x_i, \rho_X(x_i),x_{i+1}, \ldots, x_n) \big)\\
  		& = & \underbrace{ \sum_{j=2}^{i-1}(-1)^j (-1)^n\phi_{[x_1 \cdots \widehat{x_i} \cdots x_n],[x_i \cdots x_n]} \phi_{[(x_1 \ast x_i)\cdots \widehat{(x_j \ast x_i)} \cdots  (x_{i-1} \ast x_i) \rho_X(x_i)x_{i+1} \cdots x_n],[(x_j \ast x_i) \cdots (x_{i-1} \ast x_i) \rho_X{(x_i)}x_{i+1}\cdots  x_n]}}_{1A}\\
  		& &\underbrace{\Big(x_1 \ast x_i,\ldots, \widehat{x_j \ast x_i}, \ldots , x_{i-1} \ast x_i, \rho_X(x_i),x_{i+1}, \ldots, x_n \Big)}_{1A}\\
  		& + & \underbrace{(-1)^i (-1)^n\phi_{[x_1 \cdots \widehat{x_i} \cdots x_n],[x_i \cdots  x_n]} \phi_{[(x_1 \ast x_i)\cdots \cdots  (x_{i-1} \ast x_i) \widehat{\rho_X(x_i)}x_{i+1} \cdots x_n],[\rho_X{(x_i)}x_{i+1}\cdots  x_n]}}_{1B}\\
  		& &\underbrace{\Big(x_1 \ast x_i,\ldots , x_{i-1} \ast x_i, \widehat{\rho_X(x_i)},x_{i+1}, \ldots, x_n \Big)}_{1B}\\
  		&+ &\underbrace{\sum_{j=i+1}^{n}(-1)^j (-1)^n \phi_{[x_1 \cdots \widehat{x_i} \cdots x_n],[x_i \cdots  x_n]} \phi_{[(x_1 \ast x_i)\cdots  (x_{i-1} \ast x_i) \rho_X(x_i)   x_{i+1} \cdots \widehat{x_j} \cdots x_n],[(x_j \cdots  x_n]}}_{1C}\\
  		& &\underbrace{ \Big(x_1 \ast x_i,\ldots , x_{i-1} \ast x_i, \rho_X(x_i), x_{i+1}, \ldots, \widehat{x_j}, \ldots, x_n \Big)}_{1C}\\
  		& - & \underbrace{ \sum_{j=2}^{i-1}(-1)^j (-1)^n\phi_{[x_1 \cdots \widehat{x_i} \cdots x_n],[x_i \cdots  x_n]}}_{1D} \\
  		& &\underbrace{\Big((x_1 \ast x_i)\ast(x_j \ast x_i),\ldots, (x_{j-1} \ast x_i) \ast (x_j \ast x_i), \widehat{x_j \ast x_i}, x_{j+1} \ast x_i, \ldots , x_{i-1} \ast x_i, \rho_X(x_i),x_{i+1}, \ldots, x_n \Big)}_{1D}\\
  		& - &\underbrace{ (-1)^i (-1)^n \phi_{[x_1 \cdots \widehat{x_i} \cdots x_n],[x_i \cdots  x_n]}}_{1E}\\
  		& &\underbrace{\Big((x_1*x_i)*\rho_X(x_i) ,\ldots , (x_{i-1} *x_i)*\rho_X(x_i), \widehat{\rho_X(x_i)},x_{i+1}, \ldots, x_n \Big)}_{1E}\\
  		& - &\underbrace{ \sum_{j=i+1}^{n}(-1)^j (-1)^n  \phi_{[x_1 \cdots \widehat{x_i} \cdots x_n],[x_i \cdots  x_n]}}_{1F}\\
  		& &\underbrace{\Big( (x_1 \ast x_i )\ast x_j,\ldots , (x_{i-1} \ast x_i )\ast x_j, \rho_X(x_i \ast x_j), x_{i+1}\ast x_j, \ldots,{x_{j-1} \ast x_j}, \widehat{x_j}, x_{j+1}, \ldots, x_n \Big)}_{1F}\\	 	
  		& + &\underbrace{(-1)^n \phi_{[x_1 \cdots \widehat{x_i} \cdots x_n],[x_i \cdots  x_n]} \psi_{[(x_1 \ast x_i)\widehat{(x_2 \ast x_i)}\cdots  (x_{i-1} \ast x_i) \rho_X(x_i)x_{i+1} \cdots x_n],[(x_2 \ast x_i)\cdots \cdots  (x_{i-1} \ast x_i) \rho_X(x_i)x_{i+1} \cdots x_n]}}_{1G}\\
  		& &\underbrace{\Big(x_2 \ast x_i ,\ldots , x_{i-1} \ast x_i , \rho_X(x_i), x_{i+1}, \ldots, x_n \Big)}_{1G}
  	\end{eqnarray*} 
\end{small}
and 
\begin{eqnarray*}
&&   	\partial_n \big((x_1, \ldots , x_n) \big)\\
 		& = &\underbrace{ \sum_{j=2}^{i-1}(-1)^j (-1)^n \phi_{[x_1 \cdots \widehat{x_j} \cdots x_n],[x_j \cdots  x_n]}
  			(x_1,\ldots, \widehat{x_j}, \ldots , x_n)}_{2A}\\
  		& + &\underbrace{(-1)^i (-1)^n \phi_{[x_1 \cdots \widehat{x_i} \cdots x_n],[x_i \cdots  x_n]}
  			(x_1,\ldots, x_{i-1}, \widehat{x_i},x_{i+1}, \ldots, x_n)}_{2B}\\
  		&+ &\underbrace{\sum_{j=i+1}^{n}(-1)^j (-1)^n \phi_{[x_1 \cdots \widehat{x_j} \cdots x_n],[x_j \cdots  x_n]}
  			(x_1 , \ldots, \widehat{x_j}, \ldots, x_n)}_{2C}\\
  		& - &\underbrace{\sum_{j=2}^{i-1}(-1)^j (-1)^n
  		(x_1 \ast x_j,\ldots,x_{j-1} \ast x_j, x_{j+1}, \ldots , x_n)}_{2D}\\
  		& - &\underbrace{(-1)^i (-1)^n (x_1 \ast x_i,\ldots , x_{i-1} \ast x_i,x_{i+1}, \ldots, x_n)}_{2E}\\
  		& - &	\underbrace{\sum_{j=i+1}^{n}(-1)^j (-1)^n
  	(x_1 \ast x_j,\ldots , x_{j-1} \ast x_j, x_{j+1}, \ldots, x_n)}_{2F}\\
  		& + &\underbrace{(-1)^n \psi_{[x_1 \widehat{x_2}x_3 \cdots x_n],[x_2 x_3 \cdots x_n]}(x_2,\ldots , x_n)}_{2G}.
  	\end{eqnarray*} 
We see that
  	\begin{eqnarray*}
&& 1A + 2A \\
&= &\sum_{j=2}^{i-1}(-1)^j (-1)^n \Big(\phi_{[x_1 \cdots \widehat{x_i} \cdots x_n],[x_i \cdots  x_n]} \phi_{[(x_1 \ast x_i)\cdots \widehat{(x_j \ast x_i)} \cdots  (x_{i-1} \ast x_i) \rho_X(x_i)x_{i+1} \cdots x_n],[(x_j \ast x_i) \cdots (x_{i-1} \ast x_i) \rho_X{(x_i)}x_{i+1}\cdots  x_n]}\\
  		& &(x_1 \ast x_i,\ldots, \widehat{x_j \ast x_i}, \ldots , x_{i-1} \ast x_i, \rho_X(x_i),x_{i+1}, \ldots, x_n) +
  		\phi_{[x_1 \cdots \widehat{x_j} \cdots x_n],[x_j \cdots  x_n]}
  		(x_1,\ldots, \widehat{x_j}, \ldots , x_n) \Big).
  	\end{eqnarray*}
  	Then, using \eqref{consequence rack identity}, the fact that $\rho_X$ is a good involution and the condition \hyperref[axiom]{(A2)}, we see that
\begin{eqnarray*}
 && \phi_{\underbrace{[x_1 \cdots \widehat{x_i} \cdots x_n]}_{x \ast y},\underbrace{[x_i \cdots  x_n]}_{z}}\\
&&  \phi_{\underbrace{[(x_1 \ast x_i) \cdots \widehat{(x_j \ast x_i)} \cdots  (x_{i-1} \ast x_i) \rho_X(x_i) x_{i+1} \cdots x_n]}_{x},\underbrace{[(x_j \ast x_i) \cdots (x_{i-1} \ast x_i) \rho_X{(x_i)}x_{i+1}\cdots  x_n]}_{y}}\\
&=& \phi_{\underbrace{[x_1 \cdots \widehat{x_j} \cdots x_n]}_{x \ast z},\underbrace{[x_j \cdots x_n]}_{y \ast z}} \phi_{\underbrace{[x_1 \cdots \widehat{x_j} \cdots \widehat{x_i} \cdots x_n]}_{x},\underbrace{[x_i \cdots  x_n]}_{z}}.
\end{eqnarray*}
Thus, the term 1A + 2A lies in $D_{n-1}^{SR}(X,\rho_X)$.  Next, we consider
\begin{eqnarray*}
&&  1B + 2E\\
&=& (-1)^i (-1)^n \phi_{[x_1 \cdots \widehat{x_i} \cdots x_n],[x_i \cdots  x_n]} \phi_{[(x_1 \ast x_i)\cdots  (x_{i-1} \ast x_i) \widehat{\rho_X(x_i)}x_{i+1} \cdots x_n],[\rho_X{(x_i)}x_{i+1}\cdots  x_n]}\\
&& (x_1 \ast x_i,\ldots , x_{i-1} \ast x_i, \widehat{\rho_X(x_i)},x_{i+1}, \ldots, x_n)- (-1)^i (-1)^n (x_1 \ast x_i,\ldots , x_{i-1} \ast x_i,x_{i+1}, \ldots, x_n).
\end{eqnarray*}
Due to condition \hyperref[axiom]{(A7)}, we have
  	$$\phi_{\underbrace{[x_1 \cdots \widehat{x_i} \cdots x_n]}_{x \ast \rho_X(y)},\underbrace{[x_i \cdots  x_n]}_{y}} \phi_{\underbrace{[(x_1 \ast x_i)\cdots \cdots  (x_{i-1} \ast x_i) \widehat{\rho_X(x_i)}x_{i+1} \cdots x_n]}_{x},\underbrace{[\rho_X{(x_i)}x_{i+1}\cdots  x_n]}_{\rho_X(y)}}=\id,$$
and hence $1B + 2E =0$. Consider the term
\begin{eqnarray*}
&& 1C + 2C\\
& =&\sum_{j=i+1}^{n}(-1)^j (-1)^n \Big(\phi_{[x_1 \cdots \widehat{x_i} \cdots x_n],[x_i \cdots  x_n]} \phi_{[(x_1 \ast x_i)\cdots  (x_{i-1} \ast x_i) \rho_X(x_i) x_{i+1} \cdots  \widehat{x_j} \cdots x_n],[x_j \cdots  x_n]}\\
  		& &(x_1 \ast x_i,\ldots , x_{i-1} \ast x_i, \rho_X(x_i), x_{i+1}, \ldots, \widehat{x_j}, \ldots, x_n)
  		+\phi_{[x_1 \cdots \widehat{x_j} \cdots x_n],[x_j \cdots x_n]}
  		(x_1 , \ldots, \widehat{x_j}, \ldots, x_n) \Big).
  	\end{eqnarray*}
Using the condition \hyperref[axiom]{(A2)}, we see that
\begin{eqnarray*}
&& \phi_{\underbrace{[x_1 \cdots \widehat{x_i} \cdots x_j \cdots x_n]}_{x \ast z},\underbrace{[x_i \cdots x_j \cdots  x_n]}_{y \ast z}}
\phi_{\underbrace{[(x_1 \ast x_i) \cdots  (x_{i-1} \ast x_i) \rho_X(x_i) x_{i+1} \cdots \widehat{x_j} \cdots x_n]}_{x},\underbrace{[x_j \cdots  x_n]}_{z}}\\
&=& \phi_{\underbrace{[x_1 \cdots \widehat{x_j} \cdots x_n]}_{x \ast y},\underbrace{[x_j \cdots x_n]}_{z}} \phi_{\underbrace{[x_1 \cdots \widehat{x_i} \cdots \widehat{x_j} \cdots x_n]}_{x},\underbrace{[x_i \cdots \widehat{x_j} \cdots x_n]}_{y}}.
\end{eqnarray*}
Thus, 1C + 2C lies in $D_{n-1}^{SR}(X,\rho_X)$. Next, we consider
\begin{eqnarray*}
&&  1D + 2D\\
  		&=&-\sum_{j=2}^{i-1}(-1)^j (-1)^n \Big(\phi_{[x_1 \cdots \widehat{x_i} \cdots x_n],[x_i \cdots  x_n]} \\
  		& &((x_1 \ast x_i)\ast(x_j \ast x_i),\ldots, (x_{j-1} \ast x_i)*(x_j * x_i), \widehat{x_j \ast x_i}, x_{j+1} \ast x_i, \ldots , x_{i-1} \ast x_i, \rho_X(x_i),x_{i+1}, \ldots, x_n)\\
  		&+&(x_1 \ast x_j,\ldots,x_{j-1} \ast x_j, x_{j+1}, \ldots , x_n) \Big). 
  	\end{eqnarray*}
  	Since $\phi_{[x_1 \cdots \widehat{x_i} \cdots x_n],[x_i \cdots  x_n]}=\phi_{[(x_1 \ast x_j) \cdots (x_{j-1} \ast x_j) x_{j+1} \cdots \widehat{x_i} \cdots x_n],[x_i \cdots x_n]}$, it follows that 1D + 2D lies in $D_{n-1}^{SR}(X,\rho_X)$. Next, we have
  	\begin{eqnarray*}
1E + 2B &=& (-1)^i (-1)^n \Big(- \phi_{[x_1 \cdots \widehat{x_i} \cdots x_n],[x_i \cdots  x_n]}(x_1,\ldots, x_{i-1}, \widehat{\rho_X(x_i)},x_{i+1}, \ldots, x_n)\\
  		&+&\phi_{[x_1 \cdots \widehat{x_i} \cdots x_n],[x_i \cdots  x_n]}
  		(x_1,\ldots, x_{i-1}, \widehat{x_i},x_{i+1}, \ldots, x_n) \Big)\\
		&=& 0.
  	\end{eqnarray*}
 Consider the term
  	\begin{eqnarray*}
 &&1F + 2F\\
	&= &-\sum_{j=i+1}^{n}(-1)^j (-1)^n \Big(\phi_{[x_1 \cdots \widehat{x_i} \cdots x_n],[x_i \cdots x_n]}\\
  		& &(x_1 \ast x_i \ast x_j,\ldots , x_{i-1} \ast x_i \ast x_j, \rho_X(x_i \ast x_j), x_{i+1}\ast x_j, \ldots,{x_{j-1} \ast x_j}, \widehat{x_j}, x_{j+1}, \ldots, x_n)\\
  		& + &(x_1 \ast x_j,\ldots , x_{j-1} \ast x_j, x_{j+1}, \ldots, x_n) \Big).
  	\end{eqnarray*}
  	Since $\phi_{[(x_1\ast x_j) \cdots \widehat{(x_i \ast x_j)} \cdots (x_{j-1}\ast x_j) x_{j+1} \cdots x_n)],[(x_i \ast x_j) \cdots (x_{j-1}\ast x_j) x_{j+1} \cdots x_n]}=\phi_{[x_1 \cdots \widehat{x_i} \cdots x_n],[x_i \cdots x_n]}$, it follows that 1F + 2F lies in $D_{n-1}^{SR}(X, \rho_X)$. Next, consider the term
  	\begin{eqnarray*}
&& 1G + 2G\\
  		&= & (-1)^n \Big(\phi_{[x_1 \cdots \widehat{x_i} \cdots x_n],[x_i \cdots  x_n]} \psi_{[(x_1 \ast x_i)\widehat{(x_2 \ast x_i)}\cdots (x_{i-1} \ast x_i) \rho_X(x_i)x_{i+1} \cdots x_n],[(x_2 \ast x_i)\cdots \cdots  (x_{i-1} \ast x_i) \rho_X(x_i)x_{i+1} \cdots x_n]}\\
  		& &(x_2 \ast x_i ,\ldots , x_{i-1} \ast x_i , \rho_X(x_i), x_{i+1}, \ldots, x_n)
  		+\psi_{[x_1 \widehat{x_2}x_3 \cdots x_n],[x_2 x_3 \cdots x_n]}(x_2,\ldots , x_n) \Big)
  	\end{eqnarray*}
If $i>2$, then using the condition \hyperref[axiom]{(A3)}, we get 
  	\begin{eqnarray*}
  		& &\phi_{\underbrace{[x_1 \cdots \widehat{x_i} \cdots x_n]}_{x \ast y},\underbrace{[x_i \cdots  x_n]}_{z}}\\
  		& &\psi_{\underbrace{[(x_1 \ast x_i)\widehat{(x_2 \ast x_i)}\cdots (x_{i-1} \ast x_i) \rho_X(x_i)x_{i+1} \cdots x_n]}_{x},\underbrace{[(x_2 \ast x_i)\cdots \cdots  (x_{i-1} \ast x_i) \rho_X(x_i)x_{i+1} \cdots x_n]}_{y}}\\
  		&=&
  		\psi_{\underbrace{[x_1 \widehat{x_2}x_3 \cdots x_n]}_{x \ast z},\underbrace{[x_2 x_3 \cdots x_n]}_{y \ast z}}\phi_{\underbrace{[x_2 \cdots \widehat{x_i} \cdots x_n]}_{y},\underbrace{[x_i \cdots x_n]}_{z}}.
  	\end{eqnarray*}
Thus, in this case, 1G + 2G lies in $D_{n-1}^{SR}(X, \rho_X)$. If $i=2$, then the condition \hyperref[axiom]{(A9)} gives
  	\begin{eqnarray*}
  		& &-\psi_{\underbrace{[x_1 \widehat{x_2}x_3 \cdots x_n]}_{x \ast \rho_X(y)},\underbrace{[x_2 x_3 \cdots x_n]}_{y}}\\
  		&=& \phi_{\underbrace{[x_1, \widehat{x_2},x_3, \cdots x_n]}_{x \ast \rho_X(y)},\underbrace{[x_2,x_3, \cdots x_n]}_{y}}
  		\psi_{\underbrace{[x_1 \cdots x_n]}_{x},\underbrace{[\rho_X(x_2)x_3 \cdots x_n]}_{\rho_X(y)}} \eta_{\underbrace{[x_1 \widehat{x_2} x_3 \cdots x_n]}_{y}}.
  	\end{eqnarray*}
Hence, 1G + 2G lies in $D_{n-1}^{SR}(X, \rho_X)$, and we have proved that $\partial_n(V_n) \subset D_{n-1}^{SR}(X,\rho_X)$.
  \end{proof}

  \begin{proposition}\label{submoduleI}
Let $(X,\rho_X)$ be a symmetric rack. Then $\partial_n(D_{n}^{SR}(X,\rho_X)) \subset D_{n-1}^{SR}(X,\rho_X)$ for each $n \ge 1$.
  \end{proposition}
\begin{proof}
In view of Lemma \ref{partI} and Lemma \ref{partII}, it remains to verify that $\partial_1(D_{1}^{SR}(X, \rho_X))={0}$.
Note that, $D_{1}^{SR}(X, \rho_X)$ is generated by $\{\eta_{x}(x)-\rho_X(x) \mid x \in X \}$. It follows from the condition \hyperref[axiom]{(A6)} that  $$\partial_1(\eta_{x}(x)-\rho_X(x))= \eta_{x}\partial_{1}(x)-\partial_1\rho_X(x)=-\eta_x \psi_{x \ast^{-1} p,p}(p)+\psi_{\rho_X(x) \ast^{-1} p,p}(p)= 0$$
for all $x \in X$. Hence, $\partial_n(D_{n}^{SR}(X,\rho)) \subset D_{n-1}^{SR}(X,\rho)$ for all $n \ge 0$.
\end{proof}

Lemma \ref{partI}, Lemma \ref{partII} and Proposition \ref{submoduleI} lead to a (co)homology theory for symmetric racks. Let $(X,\rho_X)$ be a symmetric rack and $A$ be a right $\mathbb{Z}(X,\rho_X)$-module. Define
$$C_n \big((X,\rho_X),A\big):=A \otimes_{\mathbb{Z}(X,\rho_X)}C_n(X,\rho_X)$$	
and 
$$C_n^{SR} \big((X,\rho_X),A \big):= \big(A \otimes_{\mathbb{Z}(X,\rho_X)}C_n(X,\rho_X) \big)/ \big(A \otimes_{\mathbb{Z}(X,\rho_X)}D_n^{SR}(X,\rho_X)\big).$$
Let $\partial_n:C_n^{SR}((X,\rho_X),A) \rightarrow C_{n-1}^{SR}((X, \rho_X),A)$ be the induced $\mathbb{Z}(X,\rho_X)$-linear map. Then $\big(C_*^{SR}((X,\rho_X),A),\partial_*\big)$ is a chain complex, yielding  homology groups  $H_n^{SR}((X,\rho_X),A)$. Let us set $C_n^{SR}(X,\rho_X)= C_n(X,\rho_X)/ D_n^{SR}(X,\rho_X)$. If $A$ is a left $\mathbb{Z}(X,\rho)$-module, then we define $$C^n_{SR} \big((X,\rho_X),A \big):=\Hom_{\mathbb{Z}(X,\rho_X)} \big(C_n^{SR}(X,\rho_X),A\big)$$ and define $\delta^n: C^n_{SR}((X,\rho_X),A) \to C^{n+1}_{SR}((X,\rho_X),A)$ to be the induced coboundary map. Then, $\big(C^*_{SR}((X,\rho_X),A),\delta^*\big)$ is a cochain complex, leading to cohomology groups $H^n_{SR}((X,\rho_X),A)$.

\begin{remark}
In view of Remark \ref{Module-algebra}, for a homogeneous $(X,\rho_X)$-module $\mathscr{F}=(A,\phi,\psi,\eta)$, the cohomology group $\mathcal{H}^2_{\mathrm{SR}}(X,\mathscr{F})$  defined in Theorem \ref{theorem ext cohom} coincides with the cohomology group $H^2_{SR}((X,\rho_X),A)$ defined above.
\end{remark}

We can modify the preceding construction for quandles as well. Let $(X,\rho_X)$ be a symmetric quandle.  Define $D_n^{SQ}(X, \rho_X)$ to be the submodule of $C_n(X,\rho_X)$ generated by the $U_n \cup V_n \cup W_n$, where 
$$W_n:=\bigcup_{i=1}^{n-1} \big\{(x_1, \ldots, x_n)\in X^n \mid x_i=x_{i+1} \big\}.$$

\begin{proposition}\label{submoduleII}
	Let $(X,\rho_X)$ be a symmetric quandle. Then $\partial_n(D^{SQ}_{n}(X,\rho_X)) \subset D^{SQ}_{n-1}(X,\rho_X)$ for each $n \ge 2$.
\end{proposition}

\begin{proof}
It suffices to prove that $\partial_n(x_1, \ldots, x_{i-1},x_i,x_i,x_{i+2},\ldots, x_n) \in D^{Q}_{n-1}(X,\rho_X)$ for each $1 \leq i \leq n-1$. For $i=1$, we have
	\begin{eqnarray*}
	&& \partial_n \big((x_1,x_1,x_3, \ldots ,x_n)\big)\\
		&=&(-1)^n \phi_{[x_1x_3 \cdots x_n],[x_1 x_3 \cdots  x_n]} (x_1,x_3, \ldots , x_n)\\
		& + &\sum_{j=3}^{n}(-1)^j (-1)^n \phi_{[x_1 x_1 x_3 \cdots \widehat{x_j} \cdots x_n],[x_j \cdots  x_n]}
		(x_1,x_1,x_3, \ldots, \widehat{x_j}, \ldots , x_n)\\
		&-& (-1)^n (x_1 \ast x_1,x_3, \ldots , x_n)\\
		& - &\sum_{j=3}^{n}(-1)^j (-1)^n (x_1 \ast x_j, x_1 \ast x_j, x_3 \ast x_j, \ldots,x_{j-1} \ast x_j, x_{j+1}, \ldots , x_n)\\
		& + &(-1)^n \psi_{[x_1 \widehat{x_1}x_3 \cdots x_n],[x_1 x_3 \cdots x_n]}(x_1, x_3, \ldots , x_n)\\
		&=&\sum_{j=3}^{n}(-1)^j (-1)^n \Big(\phi_{[x_1 x_1 x_3 \cdots \widehat{x_j} \cdots x_n],[x_j \cdots  x_n]}
		(x_1,x_1,x_3, \ldots, \widehat{x_j}, \ldots , x_n)\\
		&-&(x_1 \ast x_j,x_1 \ast x_j,x_3 \ast x_j, \ldots,x_{j-1} \ast x_j, x_{j+1}, \ldots , x_n) \Big)  \in D^{Q}_{n-1}(X,\rho_X),
	\end{eqnarray*}
where the last equality follows due to
\begin{eqnarray*}
& & \phi_{[x_1x_3 \cdots x_n],[x_1 x_3 \cdots  x_n]} (x_1,x_3, \ldots , x_n)+\psi_{[x_1 \widehat{x_1}x_3 \cdots x_n],[x_1 x_3 \cdots x_n]}(x_1, x_3\ldots , x_n)\\
&=& (\phi_{[x_1x_3 \cdots x_n],[x_1x_3 \cdots  x_n]}+\psi_{[x_1 \widehat{x_1}x_3 \cdots x_n],[x_1 x_3 \cdots x_n]})(x_1, x_3, \ldots , x_n)\\
&=&(x_1, x_3, \ldots , x_n),
\end{eqnarray*}
which itself is a consequence of the condition \hyperref[Qaxiom]{(A10)}.
\par 
For $i \geq 2$, we have
\begin{eqnarray*}
&& \partial_n \big((x_1, \ldots,x_{i-1},x_i,x_i,x_{i+2},\ldots , x_n) \big)\\
	& = &\sum_{j=2}^{i-1}(-1)^j (-1)^n \Big(\phi_{[x_1 \cdots \widehat{x_j} \cdots x_n],[x_j \cdots x_n]}(x_1,\ldots, \widehat{x_j}, \ldots ,x_i,x_i,\ldots x_n)\\
	&-& (x_1 \ast x_j,\ldots,x_{j-1} \ast x_j, x_{j+1}, \ldots ,x_i,x_i,\ldots x_n) \Big)\\
	&+& (-1)^i (-1)^n \Big(\phi_{[x_1 \cdots x_{i-1}\widehat{x_i}x_ix_{i+2} \cdots x_n],[x_ix_ix_{i+2} \cdots  x_n]}(x_1,\ldots, x_{i-1}, \widehat{x_i},x_i,x_{i+2}, \ldots, x_n)\\
	&-& (x_1 \ast x_i,\ldots , x_{i-1} \ast x_i,x_i,x_{i+2}, \ldots, x_n) \Big)\\
	&+& (-1)^{i+1} (-1)^n \Big(\phi_{[x_1 \cdots x_{i-1}x_i\widehat{x_{i}}x_{i+2} \cdots x_n],[x_ix_{i+2} \cdots  x_n]}(x_1,\ldots, x_{i-1},x_i, \widehat{x_{i}},x_{i+2}, \ldots, x_n)\\
    &-& (x_1 \ast x_{i},\ldots ,x_{i-1}\ast x_i, x_{i} \ast x_{i},x_{i+2}, \ldots, x_n) \Big)\\
	&+ &\sum_{j=i+2}^{n}(-1)^j (-1)^n \Big(\phi_{[x_1 \cdots \widehat{x_j} \cdots x_n],[x_j \cdots  x_n]}(x_1 , \ldots, x_i,x_i,\ldots \widehat{x_j}, \ldots, x_n)\\
	& - &(x_1 \ast x_j,\ldots x_i\ast x_j,x_i \ast x_j,\ldots , x_{j-1} \ast x_j, x_{j+1}, \ldots, x_n) \Big)\\
	& + &(-1)^n \psi_{[x_1 \widehat{x_2}x_3 \cdots x_n],[x_2 x_3 \cdots x_n]}(x_2,\ldots, x_i,x_i,\ldots , x_n)\\
	& = &\sum_{j=2}^{i-1}(-1)^j (-1)^n \Big(\phi_{[x_1 \cdots \widehat{x_j} \cdots x_n],[x_j \cdots x_n]}(x_1,\ldots, \widehat{x_j}, \ldots ,x_i,x_i,\ldots x_n)\\
	&-& (x_1 \ast x_j,\ldots,x_{j-1} \ast x_j, x_{j+1}, \ldots ,x_i,x_i,\ldots x_n) \Big)\\
	&+ &\sum_{j=i+2}^{n}(-1)^j (-1)^n \Big(\phi_{[x_1 \cdots \widehat{x_j} \cdots x_n],[x_j \cdots  x_n]}(x_1 , \ldots,x_i,x_i,\ldots \widehat{x_j}, \ldots, x_n)\\
	& - &(x_1 \ast x_j,\ldots x_i\ast x_j,x_i \ast x_j,\ldots , x_{j-1} \ast x_j, x_{j+1}, \ldots, x_n) \Big)\\
	& + &(-1)^n \psi_{[x_1 \widehat{x_2}x_3 \cdots x_n],[x_2 x_3 \cdots x_n]}(x_2,\ldots, x_i,x_i,\ldots , x_n) \in D^Q_{n-1}(X,\rho_X),
\end{eqnarray*} 
where the $i$-th and the $(i+1)$-th terms cancel out due to the equality $$\phi_{[x_1 \cdots x_{i-1}\widehat{x_i}x_ix_{i+2} \cdots x_n],[x_ix_ix_{i+2} \cdots  x_n]}=\phi_{[x_1 \cdots x_{i-1}x_i\widehat{x_{i}}x_{i+2} \cdots x_n],[x_ix_{i+2} \cdots  x_n]}.$$
Hence, $\partial_n(D^{SQ}_{n}(X,\rho_X)) \subset D^{SQ}_{n-1}(X,\rho_X)$, and the proof is complete.
\end{proof}

We can now define the homology and cohomology of symmetric quandles. Let $(X,\rho_X)$ be a symmetric quandle and $M$ be a right $\mathbb{Z}(X,\rho_X)$-module. We define 
	$$C_n^{SQ} \big((X,\rho_X),M \big):= \big(M \otimes_{\mathbb{Z}(X,\rho_X)}C_n(X,\rho_X) \big)/ \big(M \otimes_{\mathbb{Z}(X,\rho_X)}D_n^{SQ}(X,\rho_X)\big)$$
and the boundary map to be the same as in the case of symmetric racks. Then $\big(C_*^{SQ}((X,\rho_X),M),\partial_*\big)$ is a chain complex leading to homology groups
$H_n^{SQ}((X,\rho_X),M)$.  Similarly, if $M$ is a left $\mathbb{Z}(X,\rho)$-module, then define $$C^n_{SQ} \big((X,\rho_X),M \big):=\Hom_{\mathbb{Z}(X,\rho_X)}\big(C_n^{SQ}(X,\rho_X),M \big)$$ and take $\delta^n$ to be the same as in the case of symmetric racks. Then $\big(C^*_{SQ}((X,\rho_X),M),\delta^*\big)$ is a cochain complex yielding cohomology group $H^n_{SQ}((X,\rho_X),M)$.

\begin{remark}
We can equip the algebra $\mathbb{Z}(X,\rho_X)$ with an augmentation $\e:\mathbb{Z}(X,\rho_X) \rightarrow \mathbb{Z}$, defined as $\e(\phi_{x,y})=1,~ \e(\psi_{x,y})=0$ and $\e(\eta_x)=-1$. Also, any abelian group $A$ can be considered as trivial homogeneous $(X,\rho_X)$-module by setting 
$$A_x=A, \quad  \phi_{x,y}=\id_A,  \quad \psi_{x,y}= 0 \quad \textrm{and} \quad \eta_x= -\id_A$$ for all $x, y\in X$.  In view of Remark \ref{Module-algebra}, our (co)homology is a module-theoretic generalisation of the (co)homology defined in \cite{MR2657689}, and also recovers it as a special case. Applications of this module-theoretic cohomology of symmetric racks (respectively quandles) to knot theory will be explored in a forthcoming work.
\end{remark}

\begin{remark}
By definition, a map $\kappa: C_2(X,\rho_X) \rightarrow A$, where $A$ is a $\mathbb{Z}(X,\rho_X)$-module, is a symmetric rack 2-cocycle if and only if $\kappa$ satisfies the following conditions:
		\begin{enumerate}
			\item for any $(x_1,x_2,x_3) \in X^3,$ we have $$-\phi_{[x_1x_3],[x_2x_3]}\kappa(x_1,x_3)+\phi_{[x_1x_2],x_3}\kappa(x_1,x_2)+\kappa(x_1*x_2,x_3)-\kappa(x_1*x_3,x_2*x_3)-\psi_{[x_1x_3],[x_2x_3]}\kappa(x_2,x_3)=0.$$
			\item for any $(x_1,x_2) \in X^2,$ we have
			$$\eta_{[x_1x_2]}\kappa(x_1,x_2)-\kappa(\rho_X(x_1),x_2)=0 \text{ and }
			\phi_{x_1,x_2}\kappa(x_1*x_2,\rho_X(x_2))+\kappa(x_1,x_2)=0$$
		\end{enumerate}
Furthermore,  $\kappa$ is a symmetric quandle 2-cocycle if it additionally satisfies $\kappa(x,x)=0$ for all $x \in X.$
\end{remark}
\medskip


\section{Relation between cohomology of symmetric racks and cohomology of associated groups}\label{relation cohom symm rack and group}
In \cite[Section 5]{MR1948837},  an isomorphism between the second rack cohomology of a rack and the first cohomology of its associated group has been established. In this section, we establish an analogous relation between the second cohomology of a symmetric rack and the first cohomology of its associated group.  

\begin{theorem}
Let $(X,\rho_X)$ be a symmetric rack and $A$ be an abelian group viewed as a trivial homogeneous $(X,\rho_X)$-module.  Let $A$ also be viewed as a symmetric rack  with the trivial rack operation and with $a \mapsto -a$ as the good involution, and let $\Hom(X,A)$ be the abelian group of symmetric rack homomorphisms with the point-wise addition.  Let $G_{(X,\rho_X)}$ act on $\Hom(X,A)$ from the left by $(e_x \cdot f)(z)= f(z *x)$. Then there is a group isomorphism
	$$H^2_{SR}\big((X,\rho_X),A\big) \cong H^1\big(G_{(X,\rho_X)}, \Hom(X,A)\big).$$
\end{theorem}
\begin{proof}
For $x, y, z \in X$ and $f \in \Hom(X,A)$, we see that	
$$(e_xe_y \cdot f)(z)=(e_y \cdot f)(z *x)=f((z *x)*y)=f((z*y)*(x *y))=(e_{x *y}\cdot f)(z*y)=(e_y e_{x *y} \cdot f)(z)$$
and $$(e_x^{-1}\cdot f)(z)=f(z *^{-1}x)=f(z*\rho_X(x))=(e_{\rho_X(x)}\cdot f)(z).$$
Thus, the left action of $G_{(X,\rho_X)}$ on $\Hom(X,A)$ is well-defined. 
\par

Let $f:G_{(X,\rho_X)} \rightarrow \Hom(X,A)$ be a group $1$-cocycle. Define $\theta_f:X \times X \rightarrow A$ by
$$\theta_f ((x,y))= f(e_y)(x).$$

	We claim that $\theta_f$ is a symmetric rack 2-cocycle. For all $x,y,z \in X$, we have
	\begin{eqnarray*}
		\theta_f(x,z)+\theta_f(x *z,y*z)&=& f(e_z)(x)+f(e_{y*z})(x*z)\\
		&=&f(e_z)(x)+(e_z \cdot f(e_{y*z}))(x)\\
		&=& f(e_z)(x)+(f(e_ze_{y*z})-f(e_z))(x), \quad \text{since $f$ is a group 1-cocycle} \\
		&=&f(e_ze_{y*z})(x)\\
		&=&f(e_ye_z)(x)\\
		&=&(f(e_ye_z)-f(e_y))(x)+f(e_y)(x) \\
		&=&(e_y \cdot f(e_z))(x)+f(e_y)(x), \quad \text{since $f$ is a group 1-cocycle}\\
&=&f(e_z)(x*y)+f(e_y)(x)\\
&=&		\theta_f(x*y,z)+\theta_f(x,y).
	\end{eqnarray*}
Further, since $f(e_y)$ is a symmetric rack homomorphism, we have
$$ \theta_f(x,y)+\theta_f(\rho_X(x),y) = f(e_{y})(x)+f(e_y)(\rho_X(x))= f(e_y)(x)-f(e_y)(x)=0.$$
Since $f$ is a group 1-cocycle, we get
\begin{eqnarray*}
		\theta_f(x,y)+\theta_f(x*y,\rho_X(y))&=&f(e_y)(x)+f(e_{\rho_X(y)})(x \ast y)\\
		&=&f(e_y)(x)+(e_y \cdot f(e_{\rho_X(y)}))(x)\\
		&=&f(e_y)(x)+(f(e_ye_{\rho_X(y)})-f(e_y))(x)\\
		&=&0,\quad \text{since $f(1)=0$},
	\end{eqnarray*}
and hence $\theta_f$ is a symmetric rack 2-cocycle. This gives a map $Z^1 \big(G_{(X,\rho_X)}, \Hom(X,A)\big) \rightarrow H^2_{SR} \big((X,\rho_X),A \big)$ defined by $f \mapsto [\theta_f]$. Since the group structure in both the domain and the codomain is the point-wise addition of maps, it follows that $f \mapsto [\theta_f]$ is a group homomorphism. If $f \in B^1(G_{(X,\rho_X)}, \Hom(X,A)),$ then for all $a \in G_{(X,\rho_X)}$ and $x\in X$, we have  $$f(a)(x)=(a \cdot \phi-\phi)(x)$$ for some $\phi \in \Hom(X,A)$. Thus, for all $x,y \in X$, we have 
	$$\theta_f(x,y)=f(e_y)(x)=(e_y \cdot \phi -\phi)(x)=(e_y \cdot \phi)(x)-\phi(x)=\phi(x*y)-\phi(x).$$
This shows that $\theta_f \in B^2((X,\rho_X);A),$ and hence we obtain a map $$\Psi:H^1 \big(G_{(X,\rho_X)}, \Hom(X,A)\big) \rightarrow H^2_{SR} \big((X,\rho_X),A \big).$$
\par
For the reverse direction,  let $\sigma \in Z_{SR}^2((X,\rho_X),A)$. Then $\sigma$ satisfies the symmetric rack 2-cocycle conditions
	\begin{equation}\label{cocycle condition (i)}
		\sigma(x,z)+\sigma(x*z,y*z)=\sigma(x,y)+\sigma(x*y,z),
	\end{equation}
	\begin{equation}\label{cocycle condition (ii)}
		\sigma(x,y)+\sigma(\rho_X(x),y)=0,
	\end{equation}
	\begin{equation}\label{cocycle condition (iii)}
		\sigma(x,y)+\sigma(x*y,\rho_X(y))=0,
	\end{equation}
	for all $x,y,z \in X.$ Let us define $\sigma':X \rightarrow \Hom(X,A)$ by $$\sigma'(x)(y)=\sigma(y,x)$$
for all $x, y \in X$. Since $\sigma'(x)(\rho_X(y))=\sigma(\rho_X(y),x)=-\sigma(x,y)$ by the cocycle condition \eqref{cocycle condition (ii)}, and hence $\sigma'(x) \in \Hom(X,A)$. Identifying $X$ with the set $\{e_x\mid x \in X\}$, we have a map $\tau_{\sigma}:X \rightarrow G_{(X,\rho_X)} \ltimes \Hom(X,A)$ defined by $$\tau_{\sigma}(e_x)=(e_x,\sigma'(x))$$
for all $x \in X$.	 Note that
	\begin{equation*}
		\tau_{\sigma}(e_x)\tau_{\sigma}(e_{\rho_X(x)})= (e_x,\sigma'(x))(e_{\rho_X(x)},\sigma'(\rho_X(x)))
		=\big(e_xe_{\rho_X(x)},\sigma'(x)+e_x \cdot \sigma'(\rho_X(x))\big).
	\end{equation*}
Further, the symmetric rack 2-cocycle condition \eqref{cocycle condition (iii)} gives
\begin{eqnarray*}
\big( \sigma'(x)+e_x \cdot \sigma'(\rho_X(x)) \big)(y) &=& \sigma'(x)(y)+(e_x \cdot \sigma'(\rho_X(x)))(y)\\
&=&\sigma'(x)(y)+\sigma'(\rho_X(x))(y*x)\\
&=&\sigma(y,x)+\sigma(y*x,\rho_X(x))\\
&=& 0.
\end{eqnarray*}
 Thus, $\sigma'(x)+e_x \cdot \sigma'(\rho_X(x))$ is the trivial map, and consequently $\tau_{\sigma}(e_x)\tau_{\sigma}(e_{\rho_X(x)})=(e,0),$ where $e$ is the identity of $G_{(X,\rho_X)}.$ Similarly, we can prove that $	\tau_{\sigma}(e_{\rho_X(x)})\tau_{\sigma}(e_x)=(e,0)$. Next, we consider
	\begin{eqnarray*}
&& \tau_{\sigma}(e_y) \, \tau_{\sigma}(e_{x*y}) \, \tau_{\sigma}(e_y)^{-1} \, \tau_{\sigma}(e_x)^{-1}\\
		&=&(e_y,\sigma'(y)) \, (e_{x*y},\sigma'({x*y})) \, (e_y,\sigma'(y))^{-1} \, (e_x,\sigma'(x))^{-1}\\
		&=&(e_y,\sigma'(y)) \, (e_{x*y},\sigma'({x*y})) \, (e_y^{-1},e_y^{-1}\cdot (-\sigma'(y))) \, (e_x^{-1},e_x^{-1} \cdot (-\sigma'(x)))\\
		&=& \big(e_y e_{x*y},\sigma'(y)+e_y \cdot \sigma'({x*y}) \big) \, \big(e_y^{-1}e_x^{-1},e_y^{-1}\cdot (-\sigma'(y))+e_y^{-1}e_x^{-1} \cdot (-\sigma'(x)) \big)\\
		&=&(e,\sigma'(y)+e_y \cdot \sigma'({x*y}))+ (e_ye_{x*y})\cdot (e_y^{-1}\cdot (-\sigma'(y))+e_y^{-1}e_x^{-1} \cdot (-\sigma'(x))).
	\end{eqnarray*}
The symmetric rack 2-cocycle condition \eqref{cocycle condition (i)} gives
\begin{eqnarray*}
&&	(\sigma'(y)+e_y \cdot \sigma'({x*y}))+ (e_ye_{x*y})\cdot (e_y^{-1}\cdot (-\sigma'(y))+e_y^{-1}e_x^{-1} \cdot (-\sigma'(x)))(z)\\
		&=&\sigma(z,y)+\sigma(z*y,x*y)+(e_ye_{x*y}) \cdot(e_y^{-1} \cdot (-\sigma'(y))+e_y^{-1}e_x^{-1}\cdot (-\sigma'(x)))(z)\\
		&=&\sigma(z,y)+\sigma(z*y,x*y)+(e_x\cdot (-\sigma'(y))+ (-\sigma'(x)))(z)\\
		&=&\sigma(z,y)+\sigma(z*y,x*y)+(-\sigma'(y))(z*x)-\sigma'(x)(z)\\
		&=&\sigma(z,y)+\sigma(z*y,x*y)-\sigma(z*x,y)-\sigma(z,x)\\
		&=&0.
	\end{eqnarray*}
It follows that $ \tau_{\sigma}(e_y) \tau_{\sigma}(e_{x*y}) \tau_{\sigma}(e_y)^{-1}\tau_{\sigma}(e_x)^{-1}=(e,0)$ for all $x, y \in X$, and hence we can extend $\tau_{\sigma}$ to a group homomorphism $\tau_{\sigma}:G_{(X,\rho_X)} \rightarrow G_{(X,\rho_X)} \ltimes \Hom(X,A)$.
\par 	
	Let us define $\chi_{\sigma}:G_{(X,\rho_X)} \to \Hom(X,A)$  to be the composition $$\chi_{\sigma} = p_2 \,\tau_{\sigma}$$
where $p_2:(G_{(X.\rho_X)} \ltimes \Hom(X,A)) \rightarrow \Hom(X,A)$ is the projection onto the second factor. Since $\tau_{\sigma}$ is a group homomorphism, it follows that $\chi_{\sigma}$ is a group 1-cocycle. This gives a map $Z^2_{SR}\big((X,\rho_X),A \big) \rightarrow H^1 \big(G_{(X,\rho_X)}, \Hom(X,A)\big)$  defined by $ \sigma \mapsto [\chi_{\sigma}]$. If $\pi \in B^2_{SR}((X,\rho_X),A)$, then $\pi=\delta^1( \tilde{\pi})$ for some $\tilde{\pi}\in C^1_{SR}((X,\rho_X),A)$, that is, $$\pi(x,y)=\tilde{\pi}(x)-\tilde{\pi}(x*y)$$ for all $x,y \in X$. If $g \in G_{(X,\rho_X)}$, then $g=e_{x_1}^{\e_1}e_{x_2}^{\e_2} \cdots e_{x_n}^{\e_n}$ for some $x_i \in X$ and $\e_i\in \{1,-1\}$. Thus, for all $y \in X$, we have
	\begin{eqnarray*}
		\chi_{\pi}(g)(y)&=&p_2(\tau_{\pi}(g))(y)\\
		&=&p_2(\tau_{\pi}(e_{x_1}^{\e_1}e_{x_2}^{\e_2} \cdots e_{x_n}^{\e_n}))(y)\\
		&=&p_2(\tau_{\pi}(e_{x_1})^{\e_1}\tau_{\pi}(e_{x_2})^{\e_2} \cdots \tau_{\pi}(e_{x_n})^{\e_n})(y)\\
		&=& p_2 \big((e_{x_1},\pi'(x_1))^{\e_1}(e_{x_2},\pi'(x_2))^{\e_2} \cdots (e_{x_n},\pi'(x_n))^{\e_n} \big)(y)\\
		&=&p_2((e_{x_1}^{\e_1}e_{x_2}^{\e_2} \cdots e_{x_n}^{\e_n}, \, t(x_1)+e_{x_1}^{\e_1}\cdot t(x_2))+\cdots +e_{x_1}^{\e_1}e_{x_2}^{\e_2} \cdots e_{x_n}^{\e_n} \cdot t(x_n)))(y)\\
		&=& \big(t(x_1)+e_{x_1}^{\e_1}\cdot t(x_2)+\cdots +e_{x_1}^{\e_1}e_{x_2}^{\e_2} \cdots e_{x_{n-1}}^{\e_{n-1}} \cdot t(x_n) \big)(y),
	\end{eqnarray*} 
where $t(x_i)=\pi'(x_i)$ for $\e_i=1$ and $t(x_i)=e_{x_i}^{-1} \cdot (-\pi'(x_i))$ for $\e_i=-1$. For instance, if all $\e_i=1$, then 
	\begin{eqnarray*}
		\chi_{\pi}(g)(y)&=& \pi(y,x_1)+\pi(y*x_1,x_2)+ \cdots + \pi(((y*x_1)*x_2)\cdots *x_{n-1},x_n)\\
		&=&\tilde{\pi}(y)-\tilde{\pi}(y*x_1)+\tilde{\pi}(y*x_1)-\tilde{\pi}((y*x_1*)x_2)+\cdots +\\
		&&\tilde{\pi}(((y*x_1)*x_2)\cdots *x_{n-1})-\tilde{\pi}(((y*x_1)*x_2)\cdots x_n)\\
		&=&\tilde{\pi}(y)-\tilde{\pi}((y*x_1)*x_2)\cdots x_n)\\
		&=&\tilde{\pi}(y)-(e_{x_1}^{\e_1}e_{x_2}^{\e_2} \cdots e_{x_n}^{\e_n}) \cdot \tilde{\pi}(y)\\
		&=&\tilde{\pi}(y)-g \cdot \tilde{\pi}(y)=(\tilde{\pi}-g \cdot \tilde{\pi})(y).
	\end{eqnarray*}
The cases when some $\e_i=-1$ can be dealt with in a similar way through a routine computation. Thus, $\chi_{\pi}(g)=\tilde{\pi}-g \cdot \tilde{\pi}$, which shows that $\chi_{\pi} \in B^1(G_{(X,\rho_X)}, \Hom(X,A))$. Hence, we obtain a map
$$\Psi: H^2_{SR}\big((X,\rho_X),A \big) \rightarrow H^1 \big(G_{(X,\rho_X)}, \Hom(X,A)\big)$$
and it is easy to see that $\Psi$ and $\Phi$ are inverses to each other, and hence $H^2_{SR}\big((X,\rho_X),A\big) \cong H^1 \big(G_{(X,\rho_X)}, \Hom(X,A)\big)$, which is desired.	
\end{proof}
\medskip

\begin{ack}
BK thanks CSIR for the PhD research fellowship. DS thanks IISER Mohali for the PhD research fellowship. MS is supported by the SwarnaJayanti Fellowship
grants DST/SJF/MSA-02/2018-19 and SB/SJF/2019-20/04.
\end{ack}
\medskip

\section{Declaration}
The authors declare that there is no data associated to this paper and that there are no conflicts of interests.
	

\bigskip



\begin{thebibliography}{HD}

\bibitem{MR1994219} N. Andruskiewitsch and M. Gra\~{n}a, \textit{From racks to pointed Hopf algebras}, Adv. Math. 178 (2003), no.2, 177--243.

\bibitem{MR2616383} J. M. Beck, \textit{Triples, algebras and cohomology}, Thesis (Ph.D.) Columbia University ProQuest LLC, Ann Arbor, MI, 1967. 115 pp. Republished in Repr. Theory Appl. Categ. (2003), no.2, 1--59.

\bibitem{MR1291599} F. Borceux, \textit{Handbook of categorical algebra. 1. Basic category theory}, Encyclopedia Math. Appl., 50 Cambridge University Press, Cambridge, 1994. xvi+345 pp.
	

\bibitem{MR1990571} S. Carter, D. Jelsovsky, S. Kamada, L. Langford and M. Saito, \textit{Quandle cohomology and state-sum invariants of knotted curves and surfaces}, Trans. Amer. Math. Soc. 355 (2003), no.10, 3947--3989.

\bibitem{MR4476068} N. Cazet, \textit{Surface-link families with arbitrarily large triple point number}, Topology Appl. 319 (2022), Paper No. 108234, 12 pp.

\bibitem{SDR2023} N. K. Dhanwani, H. Raundal and M. Singh, \textit{Dehn quandles of groups and orientable surfaces}, Fund. Math. 263 (2023), 167--201.

\bibitem{MR1948837} P. Etingof and M. Gra\~{n}a, \textit{On rack cohomology}, J. Pure Appl. Algebra 177 (2003), no.1, 49--59.	
	
\bibitem{MR1364012} R. Fenn, C. Rourke and B. Sanderson, \textit{Trunks and classifying spaces}, Appl. Categ. Structures (1995), no.4, 321--356.

\bibitem{MR2255194} R. Fenn, C. Rourke and B. Sanderson, \textit{The rack space}, Trans. Amer. Math. Soc. 359 (2007), no. 2, 701--740.

\bibitem{MR2945646} E. Hatakenaka and T. Nosaka, \textit{Some topological aspects of 4-fold symmetric quandle invariants of 3-manifolds}, Internat. J. Math. 23 (2012), no.7, 1250064, 31 pp.
	
\bibitem{MR2155522} N. Jackson, \textit{Extensions of racks and quandles}, Homology Homotopy Appl. 7 (2005), no.1, 151--167.

\bibitem{MR2890467} Y. Jang and K. Oshiro, \textit{Symmetric quandle colorings for spatial graphs and handlebody-links}, J. Knot Theory Ramifications (2012), no.4, 1250050, 16 pp.	
		
\bibitem{Joyce1979} D. Joyce, \textit{An algebraic approach to symmetry with applications to knot theory}, Ph.D. Thesis, University of Pennsylvania, 1979, 127 pp.

\bibitem{MR2371714} S. Kamada, 	\textit{Quandles with good involutions, their homologies and knot invariants}, Intelligence of low dimensional topology 2006, 101--108.
	
\bibitem{MR3363811} S. Kamada,	\textit{Quandles and symmetric quandles for higher dimensional knots}, Knots in Poland III. Part III, 145--158.
	Banach Center Publ., 103 Polish Academy of Sciences, Institute of Mathematics, Warsaw, 2014.
	
\bibitem{MR2657689} S. Kamada and K. Oshiro, \textit{Homology groups of symmetric quandles and cocycle invariants of links and surface-links}, Trans. Amer. Math. Soc. 362 (2010), no.10, 5501--5527.

\bibitem{KSS2023} B. Karmakar, D. Saraf and M. Singh, \textit{Cohomology of generalised Legendrian racks and state-sum invariants of Legendrian links},
https://doi.org/10.48550/arXiv.2301.06854.

\bibitem{MR3558231} V. Lebed and L. Vendramin, \textit{Homology of left non-degenerate set-theoretic solutions to the Yang--Baxter equation,} Adv. Math. 304 (2017), 1219--1261.

\bibitem{MR0217742} O. Loos, \textit{Reflexion spaces and homogeneous symmetric spaces,} Bull. Amer. Math. Soc. 73 (1967), 250--253.

\bibitem{MR0672410} S. V. Matveev, \textit{Distributive groupoids in knot theory}, in Russian: Mat. Sb. (N.S.) 119 (1) (1982) 78--88, translated in Math. USSR Sb. 47 (1) (1984), 73--83.

\bibitem{MR2821435} T. Nosaka, \textit{ 4-fold symmetric quandle invariants of 3-manifolds}, Algebr. Geom. Topol. 11 (2011), no.3, 1601--1648.

\bibitem{MR2629767} K. Oshiro, \textit{Triple point numbers of surface-links and symmetric quandle cocycle invariants},  Algebr. Geom. Topol. 10 (2010), no.2, 853--865.
	
\bibitem{MR4594919} M. Saito and E. Zappala, \textit{Extensions of augmented racks and surface ribbon cocycle invariants}, Topology Appl. 335 (2023), Paper No. 108555, 19 pp.

\bibitem{MR2699808} J. Zablow, \textit{Loops, waves, and an ``algebra'' for Heegaard splittings}, Ph.D. thesis, City Univ. of New York, 1999. 

\end{thebibliography}
\end{document}